%% file: 2-micro_at_linear_coisotropics.tex
\documentclass[12pt]{amsart}

\usepackage{amsfonts,amsmath,amsthm,amssymb}
\usepackage{graphicx}

\usepackage{marginnote}
\usepackage[margin=1in]{geometry}

\usepackage{xcolor,hyperref}
\definecolor{darkblue}{rgb}{0.0,0.0,0.5}

\definecolor{darkred}{rgb}{0.5,0.0,0.0}
\hypersetup{colorlinks,breaklinks,linkcolor=darkblue,urlcolor=darkred,anchorcolor=darkblue,citecolor=darkblue}

\colorlet{darkgreen}{green!50!black}

\usepackage{setspace}

\newcommand{\dbar}{d\hspace*{-0.08em}\bar{}\hspace*{0.1em}}

\numberwithin{equation}{subsection}

\theoremstyle{plain}
\newtheorem{theorem}{Theorem}[subsection]
\newtheorem{prop}[theorem]{Proposition}
\newtheorem{lemma}[theorem]{Lemma}
\newtheorem{corollary}[theorem]{Corollary}

\theoremstyle{definition}
\newtheorem{definition}[theorem]{Definition}
\newtheorem{remark}[theorem]{Remark}
\newtheorem{example}[theorem]{Example}
\newtheorem{condition}[theorem]{Condition}

\usepackage[all]{xy}

\usepackage{mathrsfs}

\usepackage{fancyhdr}
\allowdisplaybreaks

\begin{document}

\begin{abstract}
We develop a semiclassical second microlocal calculus of pseudodifferential operators associated to linear coisotropic submanifolds $\mathcal{C}\subset T^* \mathbb{T}^n$, where $\mathbb{T}^n = \mathbb{R}^n / \mathbb{Z}^n$.  First microlocalization is localization in phase space $T^* \mathbb{T}^n$; second microlocalization is finer localization near a submanifold of $T^* \mathbb{T}^n$.
Our second microlocal operators test distributions on $\mathbb{T}^n$ (e.g., Laplace eigenfunctions) for a coisotropic wavefront set, a second microlocal measure of absence of \emph{coisotropic regularity}.  This wavefront set tells us where, in the coisotropic, and in what directions, approaching the coisotropic, a distribution lacks coisotropic regularity.


We prove propagation theorems for coisotropic wavefront that are analogous to H\"{o}rmander's theorem for pseudodifferential operators of real principal type.  Furthermore, we study the propagation of coisotropic regularity for quasimodes of semiclassical pseudodifferential operators.  We Taylor expand the relevant Hamiltonian vector field, partially in the characteristic variables, at the spherical normal bundle of the coisotropic.  Provided the principal symbol is real valued and depends only on the fiber variables in the cotangent bundle, and the subprincipal symbol vanishes, we show that coisotropic wavefront is invariant under the first two terms of this expansion.
\end{abstract}

\title{Semiclassical second microlocalization at linear coisotropic submanifolds in the torus}
\author{Rohan Kadakia}
\date{\today}
\maketitle


2010 \emph{Mathematics Subject Classification}: 35S05, 35A18, 35A21.

\section{Introduction}

A submanifold $\mathcal{C}$ of the symplectic manifold $(T^* X,\omega)$
is said to be \emph{coisotropic} if $(T\mathcal{C})^\omega \subset T\mathcal{C}$; in words, $\mathcal{C}$ is coisotropic if the symplectic orthocomplement to its tangent bundle is a subbundle of the tangent bundle itself.

Most distributions encountered while studying PDE are not $O(h^\infty)$.  ($O(h^\infty)$ is the semiclassical analogue of $C^\infty$ smoothness.)  Some of these distributions, however, possess a different type of regularity---they are \emph{coisotropic distributions}, associated to a coisotropic submanifold $\mathcal{C}$.  We say that $u = u_h\in L^2(X)$ (uniformly as $h\downarrow 0$) is coisotropic if for all $k$ and all semiclassical pseudodifferential operators (PsDO) $A_1,\ldots,A_k\in\Psi_h(X)$ with $\sigma_\mathrm{pr}(A_j)\restriction_\mathcal{C}\equiv 0$, we have
\[
h^{-k} A_1\ldots A_k u\in L^2(X).
\]
Here, $\sigma_\mathrm{pr}(A)$ is the semiclassical principal symbol of $A$.  We may, by restricting the microsupports of the characteristic operators $A_j$, refine coisotropic regularity to be local on $\mathcal{C}$.

The central idea in this paper is that of second microlocalization at a linear coisotropic submanifold $\mathcal{C}$ of $T^*\mathbb{T}^n$, where $\mathbb{T}^n = \mathbb{R}^n / \mathbb{Z}^n$.  To (first) microlocalize is to localize in symplectic phase space, i.e., to simultaneously localize in position and momentum (up to the uncertainty principle).  Second microlocalization is more refined localization near a submanifold.  The first step is to \emph{blow up} $\mathcal{C}$, which technically means that we replace $\mathcal{C}$ with its spherical normal bundle $SN(\mathcal{C})$ (cf.\ \cite{Me-1}, \cite[Chapter~5]{Me-3}).



Next, for our second microlocal pseudodifferential calculus $\Psi_{2,h}(\mathcal{C})$, the collection of symbols consists of functions that are smooth on the blown up space.  In particular, this includes functions singular on the original (i.e., blown down) space with conormal singularities, resolved in the blowup.  Thus, $\Psi_{2,h}(\mathcal{C})$ contains the ordinary pseudodifferential calculus $\Psi_h(\mathbb{T}^n)$ as a subalgebra\footnote{Technically, as we see later, only PsDO with compactly supported symbols may be regarded as elements of $\Psi_{2,h}(\mathcal{C})$.}.

Associated to the calculus $\Psi_{2,h}(\mathcal{C})$ is a wavefront set ${}^2\mathrm{WF}$ in $SN(\mathcal{C})$, whose absence (together with absence of standard semiclassical wavefront $\mathrm{WF}_h$ in $T^*\mathbb{T}^n\backslash\mathcal{C}$) is equivalent to $u$ being a coisotropic distribution.  A helpful analogy is that presence of second wavefront set in $SN(\mathcal{C})$ is to failure of coisotropic regularity in $\mathcal{C}$ as presence of homogeneous wavefront in the cotangent bundle is to singular support in the base.  ${}^2\mathrm{WF}$ tells us where, in the coisotropic, \emph{and} in what directions, approaching the coisotropic, a distribution lacks coisotropic regularity; for instance, see Example \ref{ex:second wavefront}.

Whether in the homogeneous or semiclassical setting, several instances of second microlocalization exist in the literature: A.\ Vasy and J.\ Wunsch in the special case of Lagrangian submanifolds \cite{VW}; N.\ Anantharaman, C.\ Fermanian, and F.\ Maci\`{a} in \cite{AM-1,AM-2,AFM} to study \emph{defect measures}; and J-M.\ Bony's \cite{Bon} second microlocalization at conic Lagrangians.  To study resonances, J.\ Sj\"{o}strand and M.\ Zworski \cite{SZ} construct a second microlocal calculus for hypersurfaces.  Additional sources are \cite{BoLe,Del,CFK,KaKa,KaLa,Lau,Leb}.

Every coisotropic submanifold is endowed with a \emph{characteristic foliation}.  The leaves of the characteristic foliation of $\mathcal{C}$ are the integral curves of the Hamiltonian vector fields of its defining functions.  All of our results pertain specifically to linear coisotropic submanifolds.  As coordinates on $(T^* \mathbb{T}^n,\omega)$, we take $(x,\xi)$, where $\omega = d\xi\wedge dx$.  Then, a $d$-codimensional linear coisotropic is of the form $\mathcal{C} = \mathbb{T}^n \times \{\mathbf{v}_1\cdot\xi = \ldots = \mathbf{v}_d\cdot\xi = 0\}$ for $\mathbf{v}_j\in\mathbb{R}^n$.

\subsection{Propagation of coisotropic second wavefront set}
First, we show by commutator methods the analogue of H\"{o}rmander's real principal type theorem \cite{Ho-2}:

\begin{theorem} \label{primary propagation simple}
Let $P \in \Psi_{2,h}(\mathcal{C})$ and suppose that $P$ has real valued second principal symbol ${}^2 \sigma_\mathrm{pr}(P)$.  Assume also that the distribution $u$ satisfies $Pu = f$.  Then ${}^2 \mathrm{WF}(u) \backslash {}^2 \mathrm{WF}(f)$ is invariant under Hamiltonian flow at $SN(\mathcal{C})$.
\end{theorem}


For the next result, we consider an ordinary semiclassical PsDO $P\in\Psi_h(\mathbb{T}^n)$, regarded as an element of the second microlocal calculus.  Further, we require that the principal symbol of $P$ depends only on the fiber variables $\xi$, and that its subprincipal symbol vanishes.  We calculate the Hamiltonian vector field $H$ for ${}^2 \sigma_\mathrm{pr}(P)$.  Then we Taylor expand $H$ at $SN(\mathcal{C})$: if $SN(\mathcal{C}) = \{\rho=0\}$, then $H = H_1 + \rho H_2 + O(\rho^2)$.

\begin{theorem} \label{secondary propagation simple}
Assume $P\in\Psi_h(\mathbb{T}^n)$ has real principal symbol depending only on the fiber variables in the cotangent bundle, that its subprincipal symbol vanishes, and that $Pu = O_{L^2(\mathbb{T}^n)}(h^\infty)$.  Then ${}^2 \mathrm{WF}(u)\cap SN(\mathcal{C})$ is invariant under both $H_1$ and $H_2$.
\end{theorem}

\noindent This crucially hinges on $P$ being an ordinary semiclassical PsDO, so having a total symbol that is smooth even on the blown down space $T^* \mathbb{T}^n$.

The author is supported in part by NSF RTG grant 1045119.  He would like to thank Jared Wunsch for introducing him to the problems addressed in this paper, and for helpful discussions throughout.  The author is grateful also to Dean Baskin and Alejandro Uribe.  Much of this work was carried out at Northwestern University.

\section{Preliminaries}

Let $(M,\omega)$ be a $2n$-dimensional symplectic manifold.  We will be interested in the case $M = T^* \mathbb{T}^n$.  For a submanifold $\mathcal{C}\subset M$, consider the \emph{symplectic orthocomplement} $(T\mathcal{C})^\omega$ of the tangent bundle $T\mathcal{C}$, defined as the union of its fibers:
\[
(T_p \mathcal{C})^\omega := \{ v\in T_p M \ | \ \omega(v,w)=0 \ \mathrm{for \ all} \ w\in T_p \mathcal{C} \}.
\]

\begin{definition}[Coisotropic submanifold]
$\mathcal{C}$ is said to be a \emph{coisotropic} submanifold of $M$ if $(T\mathcal{C})^\omega$ is a subbundle of $T\mathcal{C}$; that is, $\mathcal{C}$ is coisotropic if $(T_p \mathcal{C})^\omega$ is a subspace of $T_p \mathcal{C}$ for each $p\in\mathcal{C}$.
\end{definition}

In particular, this means that $\dim{(\mathcal{C})}\geq n$.  If $\dim{(\mathcal{C})} = n$, then $\mathcal{C}$ is a Lagrangian submanifold.

\subsection{Coisotropic regularity}

For $m,k\in\mathbb{R}$, let $\Psi^{m,k}_h(\mathbb{T}^n) = h^{-k} \Psi^m_h(\mathbb{T}^n)$ be the space of semiclassical pseudodifferential operators of differential order $m$.  For treatments of the semiclassical pseudodifferential calculus, see \cite{DS,Ma,Zw}.  Let
\begin{equation} \label{symbol inequality}
S^m (T^* \mathbb{T}^n) := \{a(x,\xi)\in C^\infty (T^* \mathbb{T}^n) \ | \ \forall\alpha,\beta \ \exists C_{\alpha\beta}>0 \ \mathrm{such \ that} \ |\partial^\alpha_x \partial^\beta_\xi a| \leq C_{\alpha\beta}\left\langle \xi \right\rangle^{m-|\beta|}\}.
\end{equation}
The symbol class $S^m$ is due to J.J.\ Kohn and L.\ Nirenberg \cite{KN}.  Then $\Psi^{m,k}_h(\mathbb{T}^n)$ consists \emph{locally} of quantizations of symbols in $h^{-k} C^\infty ([0,1)_h;S^m(T^* \mathbb{T}^n))$.

Let $A\in\Psi^{m,k}_h(\mathbb{T}^n)$.  We will generally not be interested in the differential order of $A$, which corresponds to the behavior of its total symbol at infinity in the fibers of the cotangent bundle.  Thus, we define the subalgebra $\widetilde{\Psi}^k_h(\mathbb{T}^n) \subset \Psi^{-\infty,k}_h(\mathbb{T}^n)$ consisting locally of quantizations of $h^{-k} C^\infty_c (T^* \mathbb{T}^n \times [0,1)_h)$ (i.e., the total symbols are compactly supported in the fibers) plus quantizations of $h^\infty C^\infty([0,1);S^{-\infty}(T^* \mathbb{T}^n))$ (symbols residual in both semiclassical and differential filtrations).  This is the same subalgebra considered in the motivating paper \cite{VW}.


\subsubsection{Characteristic operators}

Let $\mathcal{C}$ be any coisotropic submanifold of $T^* \mathbb{T}^n$.  Consider
\[
\mathfrak{M}_\mathcal{C} := \{A\in\Psi_h^0(\mathbb{T}^n) \ | \ \sigma_0(A)\restriction_\mathcal{C} = 0\}.
\]
We call this the module of \emph{characteristic} operators associated to $\mathcal{C}$.  An application of Taylor's theorem proves that $\mathfrak{M}_\mathcal{C}$ is finitely generated.  It is closed under commutators, as well.  To see this, take $A,B\in\mathfrak{M}_\mathcal{C}$.
Recall that $[A,B]$ has principal symbol $\frac{h}{i}\{\sigma_0(A),\sigma_0(B)\}$.  Since $\{\sigma_0(A),\sigma_0(B)\} = \mathcal{L}_{H_{\sigma_0(A)}}(\sigma_0(B))$, $H_{\sigma_0(A)}$ is tangent to $\mathcal{C}$, and $\sigma_0(B)$ is constant on $\mathcal{C}$, then $\{\sigma_0(A),\sigma_0(B)\}\restriction_\mathcal{C} = 0$.

As an example, if $\mathcal{C} = \{\xi_1=\ldots=\xi_d=0\}$, the module $\mathfrak{M}_\mathcal{C}$ is generated by the differential operators $h D_{x_j}$, $1\leq j\leq d$.

\begin{remark}
If we write $u_h\in L^2(\mathbb{T}^n)$, we mean that $u_h$ lies in $L^2$ uniformly in $h$ as $h\downarrow 0$.  We will suppress $h$ dependence of families of distributions.  Likewise, we will write $P$ instead of $P_h$ when referring to the family of operators $P_h$.
\end{remark}

\subsubsection{Definition of coisotropic regularity}

\begin{definition}[Coisotropic regularity]
Let $\mathcal{C}$ be a coisotropic submanifold of $T^* \mathbb{T}^n$.  A distribution $u\in L^2(\mathbb{T}^n)$ is said to exhibit \emph{coisotropic regularity} with respect to $\mathcal{C}$ at the point $(x,\xi)\in\mathcal{C}$ if $(x,\xi)$ has a neighborhood $U\subset T^* \mathbb{T}^n$ such that for all $P_j\in\mathfrak{M}_\mathcal{C}$ microsupported in $U$, we have
\begin{equation} \label{iterated regularity condition}
P_1\ldots P_k u\in h^k L^2(\mathbb{T}^n)
\end{equation}
(for all $k$).

Suppose that $u$ has coisotropic regularity everywhere on $\mathcal{C}$.  Then we call $u$ a \emph{coisotropic distribution} or simply say that $u$ is coisotropic.  Or, if condition \eqref{iterated regularity condition} holds only for $k_0\leq k$ for some $k$, then $u$ has coisotropic regularity of order $k$ at $(x,\xi)$.
\end{definition}

\begin{remark}
(1) We may write the defining condition for coisotropic regularity of $u\in L^2(\mathbb{T}^n)$ equivalently as
\[
(h^{-1} P_1) \ldots (h^{-1} P_k) u \in L^2(\mathbb{T}^n), \ \forall k, \ P_j\in\mathfrak{M}_\mathcal{C}.
\]
Note that since each $P_j$ is a semiclassical PsDO of order 0 in $h$, it is a bounded operator on $L^2$, so necessarily $P_1 \ldots P_k u \in L^2(\mathbb{T}^n)$.  By contrast, $h^{-1} P_j\in\Psi^1_h(\mathbb{T}^n)$ will generally \emph{not} be $L^2$ bounded.

(2) We are certainly not the first to define a notion of regularity by means of iterated application of characteristic operators.  The original iterated regularity characterization, of conic Lagrangian distributions, is given by L.\ H\"{o}rmander and R.\ Melrose \cite[Section~25.1]{Ho85}.
\end{remark}

\begin{example}
To determine whether some $u\in L^2(\mathbb{T}^3)$ is (globally) coisotropic with respect to $\mathcal{C} = \{\xi_1 = \xi_2 = 0,\xi_3\in\mathbb{R}\}$, amounts to checking for $L^2$-Sobolev regularity \emph{in the directions} $x_1$ and $x_2$.  For instance, $e^{i x_3 / h}$ is a coisotropic distribution with respect to $\mathcal{C}$, but $e^{i x_1 / h}$ is nowhere coisotropic at $\mathcal{C}$.
\end{example}

\subsubsection{Coisotropic regularity, extended}

For $s\in\mathbb{R}$, we may consider the set of distributions that are coisotropic of order $k\in\mathbb{Z}_{\geq 0}$ relative to $h^s L^2(\mathbb{T}^n)$:
\begin{equation} \label{eq:distributions}
I^k_{(s)}(\mathcal{C}) := \{u \ | \ h^{-j-s} A_1 \ldots A_j u \in L^2(\mathbb{T}^n) \ \forall A_i\in\mathfrak{M}_\mathcal{C}, 0\leq j\leq k\}.
\end{equation}
However, because $\mathfrak{M}_\mathcal{C}$ is (locally) finitely generated, let $\{B_1,\ldots,B_J\}$ be a generating set.  To check whether $u$ lies in $I^k_{(s)}(\mathcal{C})$, it suffices to check $h^{-|\beta|-s} \mathbf{B}^\beta u \in L^2(\mathbb{T}^n)$, where $\mathbf{B}^\beta = B_1^{\beta_1} \cdots B_J^{\beta_J}$ and $0\leq |\beta|\leq k$.

We may extend the definition of $I^k_{(s)}(\mathcal{C})$ to $k\in\mathbb{R}$ by interpolation (to $k\in\mathbb{R}_{\geq 0}$) and duality (to negative $k$).  For each $s$, $I^\infty_{(s)}(\mathcal{C})$ is a space of semiclassical coisotropic distributions.

\subsection{Real blowup of submanifolds}

For a thorough treatment of this topic, see R.\ Melrose's notes in \cite{Me-1} and \cite[Chapter~5]{Me-3}.  If $M$ is a manifold without boundary and $Y$ is any submanifold of $M$, then $Y$ blown up in $M$ is $[M;Y] = M \backslash Y \cup SN(Y)$.  Here, $SN(Y)$ is the spherical (unit) normal bundle of $Y$.  $[M;Y]$ is a manifold with boundary, and the boundary of $[M;Y]$ is $SN(Y)$.

Next, suppose $M$ is a manifold with boundary and $Y$ is a submanifold \emph{lying in the boundary of $M$}.  Then the blowup of $M$ along $Y$ is $[M;Y] = M \backslash Y \cup SN^+ (Y)$, where $SN^+ (Y)$ is the \emph{inward pointing} part of the spherical normal bundle to $Y$ in $M$.  $SN^+ (Y)$ is the \emph{front face} of the blowup, and the \emph{side face} is $\partial M \backslash Y$.  There is a smooth \emph{blowdown map} $\beta: [M;Y] \longrightarrow M$ projecting the front face to $Y$; $\beta$ is a diffeomorphism away from the front face.  Thus, crucially, there are more smooth functions on the blown up space than on the original manifold.  In our case, $M = T^* \mathbb{T}^n \times [0,1)_h$ (which has boundary $T^* \mathbb{T}^n \times \{h=0\}$) and $Y = \mathcal{C} \times \{h=0\}$.

\subsection{Linear coisotropics}

We study linear coisotropics in $T^* \mathbb{T}^n$, defined as follows:

\begin{definition}
A $d$-codimensional \emph{linear coisotropic submanifold} $\mathcal{C}\subset T^* \mathbb{T}^n$ has the form
\[
\mathcal{C} = \mathcal{C}(\mathbf{v}_1,\ldots,\mathbf{v}_d) = \mathbb{T}^n_x \times \{\mathbf{v}_1\cdot\xi = \ldots = \mathbf{v}_d\cdot\xi = 0\},
\]
where $\{\mathbf{v}_1,\ldots,\mathbf{v}_d\}\subset\mathbb{R}^n$ is linearly independent over $\mathbb{R}$.
\end{definition}

A simple linear coisotropic is $\{\xi_1 = \ldots = \xi_d = 0\}$.  In fact, locally every coisotropic is of this form \cite[Theorem~21.2.4]{Ho85}.

\subsection{Second microlocal symbols}

The domain of our total symbols is the manifold with corners $S_\mathrm{tot} := [T^* \mathbb{T}^n \times [0,1)_h;\mathcal{C} \times \{h=0\}]$.  The corner occurs at the intersection of the front face $SN^+(\mathcal{C}\times\{h=0\})$ and the side face.  Second principal symbols (defined in Section \ref{sec:second principal}) live on $S_\mathrm{pr} := [T^* \mathbb{T}^n;\mathcal{C}]$, which is identifiable with the side face of $S_\mathrm{tot}$.

\begin{example}
In Figure \ref{fig:totsymb}, we see the total symbol space for the coisotropic $\mathcal{C} = \mathbb{T}^2 \times \{\xi_1\in\mathbb{R},\xi_2 = 0\}$, with base variables omitted.  Notice that the front face is a \emph{half} cylinder.  Since we are not interested in negative values of $h$, only (unit) normal vectors pointing into the interior of $T^* \mathbb{T}^2 \times [0,1)_h$ are considered.
\end{example}

\begin{figure}[htb]
 \centering 
 \def\svgwidth{200pt}
 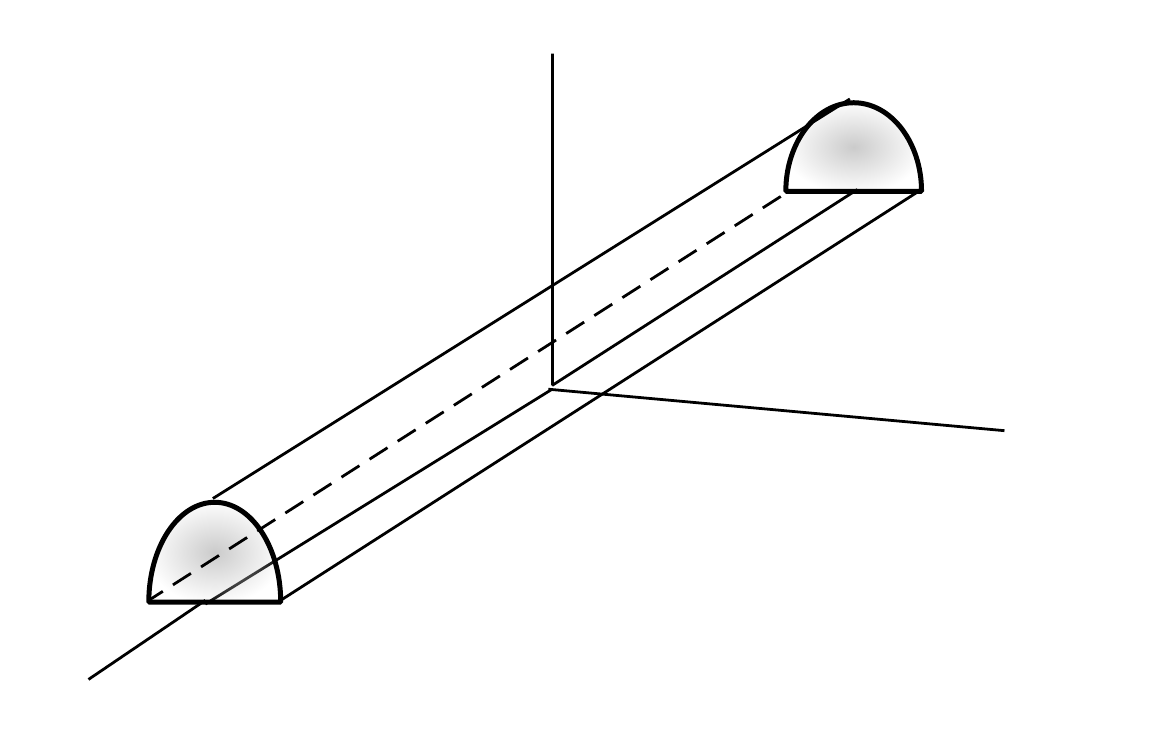
 \caption{$S_\mathrm{tot}$ for $\mathcal{C} = \{\xi_2 = 0\}$}
 \label{fig:totsymb} 
\end{figure}

In $S_\mathrm{pr}$, we replace $\mathcal{C}\subset T^*\mathbb{T}^n$ with its \emph{full} spherical normal bundle.  Let $\rho_{\mathrm{ff}}$ and $\rho_{\mathrm{sf}}$ denote boundary defining functions for the front and side faces of $S_\mathrm{tot}$, respectively.

\begin{definition}[Total symbols]
For $m,l\in\mathbb{R}$, let $$S^{m,l}(S_\mathrm{tot}) := \rho_{\mathrm{sf}}^{-m} \rho_{\mathrm{ff}}^{-l} C_{c}^{\infty}(S_\mathrm{tot}).$$
\end{definition}

We define
\[
S^{-\infty,l}(S_\mathrm{tot}) := \bigcap_{m\in\mathbb{R}} S^{m,l}(S_\mathrm{tot}).
\]

\begin{remark}
Here we clarify that we are considering \emph{two different} spherical normal bundles.  First, there is the inward-pointing spherical normal bundle $SN^+ (\mathcal{C}\times\{h=0\})$, which is the front boundary face of $S_\mathrm{tot}$.  Second, there is the full spherical normal bundle to $\mathcal{C}\subset T^* \mathbb{T}^n$, namely $SN(\mathcal{C})$.  The principal symbol space $S_\mathrm{pr}$ is a manifold with boundary, with $\partial S_\mathrm{pr} = SN(\mathcal{C})$.  Finally, $SN(\mathcal{C})$ is identifiable with the corner of $S_\mathrm{tot}$.
\end{remark}

Let $\dbar\xi := (2\pi h)^{-n} d\xi$.  Let ${}^h \mathrm{Op_l}$, ${}^h \mathrm{Op_W}$, and ${}^h \mathrm{Op_r}$ represent semiclassical left, Weyl, and right quantization, respectively: For $a\in S^{m,l}(S_\mathrm{tot})$,
\begin{align*}
{}^h \mathrm{Op_l}(a) &= \int{e^{\frac{i}{h}(x-y)\cdot\xi} \chi(x,y) a(x,\xi;h) \ \dbar\xi}; \\
{}^h \mathrm{Op_W}(a) &= \int{e^{\frac{i}{h}(x-y)\cdot\xi} \chi(x,y) a\left(\frac{x+y}{2},\xi;h\right) \dbar\xi}; \ \mathrm{and} \\
{}^h \mathrm{Op_r}(a) &= \int{e^{\frac{i}{h}(x-y)\cdot\xi} \chi(x,y) a(y,\xi;h) \ \dbar\xi}.
\end{align*}
Also, if $a\in C^\infty(\mathbb{T}^n;S^{m,l}(S_\mathrm{tot}))$, then
\begin{equation} \label{general quantization}
I_h(a) := \int{e^{\frac{i}{h}(x-y)\cdot\xi}\chi(x,y)a(x,y,\xi;h) \ \dbar\xi}.
\end{equation}

$\chi$ is any cutoff function supported in a neighborhood of the diagonal, and $\chi\equiv 1$ in a smaller neighborhood of the diagonal.  The purpose of $\chi$ is to make sense of the difference $(x-y)$ appearing in the phase.  A stationary phase argument shows that the choice of a particular $\chi$ does not matter, up to $O(h^\infty)$.


Our calculus will be denoted $\Psi_{2,h}(\mathcal{C})$.  The calculus will consist of, say, left quantizations of elements of $S^{m,l}(S_\mathrm{tot})$ ($m,l\in\mathbb{R}$), as well as \emph{residual operators} to be introduced in Section \ref{section on residual ops}.

\section{\textsc{Composition and Invariance} \label{section on residual ops}}

Let $\mathcal{C} = \mathbb{T}^n \times \{\mathbf{v}_1\cdot\xi = \ldots = \mathbf{v}_d\cdot\xi = 0\}$, with $\mathbf{v}_i\in\mathbb{R}^n$ linearly independent.

\subsection{Residual operators}

We next define the ``residual'' elements which, along with quantizations of symbols, comprise the second microlocal calculus $\Psi_{2,h}(\mathcal{C})$.  Residual operators play the same role as smoothing operators in the $C^\infty$ pseudodifferential calculus.

Let $B_j = \mathbf{v}_j \cdot h D_x$, where $D_x = (D_{x_1} \ \cdots \ D_{x_n})^t$.  Then $B_1,\ldots,B_d\in\Psi^0_h(\mathbb{T}^n)$ generate the module $\mathfrak{M}_\mathcal{C}$ of operators characteristic on $\mathcal{C}$, and let $\mathbf{B}^\beta := B_1^{\beta_1}\cdots B_d^{\beta_d}$ be a monomial formed from these generators.  Also, let $\widetilde{B}_j = \mathbf{v}_j \cdot D_x$ and $\widetilde{\mathbf{B}}^\beta = \widetilde{B}_1^{\beta_1}\cdots\widetilde{B}_d^{\beta_d}$.

\begin{definition}[Residual operator]
For $l\in\mathbb{R}$, the bounded linear operator $R$ is in the residual space $\Re^l$ if: \begin{condition} \label{involutizing condition}
$R$ is \emph{involutizing}: for $u_h\in L^2(\mathbb{T}^n)$ and multi-indices $\beta,\gamma$, $R$ satisfies
\[
h^{-|\beta+\gamma| + l}\left(\mathbf{B}^\beta R \mathbf{B}^\gamma\right) u_h \in L^2(\mathbb{T}^n).
\]
(Since we are composing $\mathbf{B}$ on the right of $R$, if this condition is fulfilled, then the adjoint operator $R^*$ is also involutizing.)
\end{condition}
\end{definition}

Let $\Re := \bigcup_{l\in\mathbb{R}} \Re^l$.  Then:

\begin{definition}
The second microlocal calculus associated to $\mathcal{C}$ is
\[
\Psi_{2,h}(\mathcal{C}) := \Re + \bigcup_{m,l\in\mathbb{R}}{{}^h \mathrm{Op_l}(S^{m,l}(S_\mathrm{tot}))}.
\]
\end{definition}

We will show that $\Psi_{2,h}(\mathcal{C})$ is closed under composition.  The key result is a reduction theorem, \autoref{composition theorem}.  The calculus is also closed under asymptotic summation: if $A_j\in\Psi^{m-j,l}_{2,h}(\mathcal{C})$, then there exists $A\in\Psi^{m,l}_{2,h}(\mathcal{C})$ for which
\[
A - \sum_{j=0}^{N-1} A_j \in \Psi^{m-N,l}_{2,h}(\mathcal{C})
\]
for all $N$.

\subsection{Reduction and Composition}

\begin{theorem}[Right reduction] \label{composition theorem}
Let $a(x,y,\xi;h)\in C^\infty(\mathbb{T}^n;S^{m,l}(S_\mathrm{tot}))$.  Then there exists $b\in S^{m,l}(S_\mathrm{tot})$ such that $I_h(a)={}^h \mathrm{Op_r}(b)+R$, and the remainder $R$ belongs to $\Re^l$.
\end{theorem}

Before we prove this theorem, we state and prove a result concerning boundedness.

\begin{prop}[$L^2$ boundedness] \label{boundedness prop}
Let $s\in C^{\infty}_{c}(\mathbb{T}^n \times S_\mathrm{tot})$.  Then $I_h(s)\in\Psi^{0,0}_{2,h}(\mathcal{C})$ is bounded on $L^2(\mathbb{T}^n)$.
\end{prop}

This parallels the standard result that $h$-pseudodifferential operators in $\Psi^0_h(\mathbb{T}^n)$ (here, $0$ is the order in the $h$-filtration, not the differential order) are $L^2$ bounded.  The following is an essential ingredient of our proof of $L^2$ boundedness.

\begin{lemma}[Calder\'{o}n--Vaillancourt theorem]
Let $a(x,y,\xi;h)\in C_c^\infty(\mathbb{T}^n\times \mathbb{T}^n\times\mathbb{R}^n\times[0,1))$.  Suppose that for all multi-indices $\alpha,\beta$, there exists a constant $C_{\alpha\beta} > 0$ (independent of $h$) such that
\begin{equation} \label{C-V estimate}
|\partial^{\alpha}_{x,y} \partial^{\beta}_{\xi} a(x,y,\xi;h)|\leq C_{\alpha\beta}.
\end{equation}
Then
\[
I(a) = (2\pi)^{-n} \int e^{i(x-y)\cdot\xi} \chi(x,y) a(x,y,\xi;h)\ d\xi
\]
is bounded, as an operator on $L^2(\mathbb{T}^n)$.  (Note that this is a non-semiclassical quantization.)
\end{lemma}

\begin{proof}
See \cite{CaVa}.
\end{proof}

\begin{proof}[Proof of Proposition \ref{boundedness prop}]
Let $s\in C^\infty_c(\mathbb{T}^n \times S_\mathrm{tot})$ and change variables $\eta = \xi/h$.  Then the estimate in \eqref{C-V estimate} becomes
\begin{equation} \label{estimate in C-V}
h^{|\beta|} \left|\left(\partial^{\alpha}_{x,y} \partial^{\beta}_{\xi} s\right)(x,y,\xi;h)\right| = |\partial^{\alpha}_{x,y} \partial^{\beta}_{\eta}[s(x,y,h\eta,h)]| \leq C_{\alpha\beta}.
\end{equation}

Using a partition of unity, we may decompose $s$ into pieces supported on $\mathbb{T}^n$ times the lift to $S_\mathrm{tot}$ of $\{\mathbf{v}_j\cdot\xi \neq 0\} \times [0,1)_h \subset T^* \mathbb{T}^n\times [0,1)_h$ for $j=1,\ldots,d$.  By symmetry, it is enough to study the part of $s$ supported on $\mathbb{T}^n$ times the lift of $\{\mathbf{v}_1\cdot\xi \neq 0\} \times [0,1)$.

We may extend the linearly independent set $\{\mathbf{v}_1,\ldots,\mathbf{v}_d\}$ to a basis
\[
\{\mathbf{v}_1,\ldots,\mathbf{v}_d,\mathbf{w}_{d+1},\ldots,\mathbf{w}_n\}
\]
for $\mathbb{R}^n$.  Locally, in a neighborhood of the corner of $S_\mathrm{tot}$, we employ the coordinates $x$, $y$, $H = h / {(\mathbf{v}_1\cdot\xi)}$, $\zeta = \mathbf{v}_1\cdot\xi$,
\[
\Xi = (\Xi_2,\ldots,\Xi_d) = \left(\frac{\mathbf{v}_2\cdot\xi}{\mathbf{v}_1\cdot\xi},\ldots,\frac{\mathbf{v}_d\cdot\xi}{\mathbf{v}_1\cdot\xi}\right),
\]
and $\mathbf{W} = (\mathbf{w}_{d+1}\cdot\xi,\ldots,\mathbf{w}_n\cdot\xi)$.

\emph{We lift $h^{|\beta|} \partial^{\alpha}_{x,y} \partial^{\beta}_{\xi}$ to the coordinates just introduced, then show that $s$ satisfies the estimate in \eqref{estimate in C-V} under application of the lifted vector field.}  In particular, we will lift $h \partial_{\xi_i}$ for $1\leq i\leq n$.  We have

\begin{equation*}
\partial_{\xi_i} = v_1^i \partial_\zeta - \frac{h v_1^i}{\zeta^2} \partial_H + \mathbf{w}^i\cdot\partial_{\mathbf{W}} + \frac{\partial \Xi}{\partial \xi_i}\cdot\partial_\Xi,
\end{equation*}
where $\mathbf{w}^i = \left(w_{d+1}^i,\ldots,w_n^i\right)$.  Let $\mathbf{v}^i = (v_2^i,\ldots,v_d^i)$.  We calculate that
\[
\frac{\partial \Xi}{\partial \xi_i}  = \left(\frac{v_2^i \zeta - v_1^i(\mathbf{v}_2\cdot\xi)}{\zeta^2},\ldots,\frac{v_d^i \zeta - v_1^i(\mathbf{v}_d\cdot\xi)}{\zeta^2}\right).
\]

Therefore,
\begin{equation} \label{hard lift}
h \partial_{\xi_i} = H \left(v_1^i \zeta \partial_\zeta - v_1^i H \partial_H + \mathbf{w}^i \cdot \zeta \partial_{\mathbf{W}} + \mathbf{v}^i \cdot \partial_\Xi - v_1^i \Xi\cdot\partial_\Xi\right) =: H \vec{V}.
\end{equation}

For $1\leq i\leq n$, $|h \partial_{\xi_i} s|$ is thus bounded (independently of $h$), because $s\in C^\infty_c(\mathbb{T}^n \times S_\mathrm{tot})$ is compactly supported in the variables $\zeta$, $H$, $\mathbf{W}$, and $\Xi$.

Therefore, since \eqref{estimate in C-V} is satisfied, we may apply the Calder\'{o}n--Vaillancourt theorem to conclude that $I_h(s):L^2(\mathbb{T}^n)\rightarrow L^2(\mathbb{T}^n)$.
\end{proof}

Let $\mathcal{V}_b$ represent the set of vector fields on $S_\mathrm{tot}$ tangent to both front and side faces.  Notice that $\vec{V}\in\mathcal{V}_b$.  Now we prove \autoref{composition theorem}.

\begin{proof}
We use the same coordinates as in the previous proof, and let $\rho_\mathrm{sf} = H = h / {(\mathbf{v}_1\cdot\xi)}$, $\rho_\mathrm{ff} = \zeta = \mathbf{v}_1\cdot\xi$.

Let $a(x,y,\xi;h)\in C^\infty(\mathbb{T}^n;S^{m,l}(S_\mathrm{tot}))$.  Taylor's formula yields the asymptotic sum
\[
a(x,y,\xi;h)\sim\sum_\alpha{\frac{(x-y)^\alpha}{\alpha!}(\partial^{\alpha}_{x}a)(y,y,\xi;h)}.
\]

We see from \eqref{hard lift} that for all $1\leq j\leq n$, application of $h \partial_{\xi_j}$ to a second microlocal total symbol improves its decay at the side face.  Therefore, for any multi-index $\alpha$,
\[
(h\partial_\xi)^\alpha(\partial^{\alpha}_{x}a)(y,y,\xi;h)\in H^{-m+|\alpha|} \zeta^{-l} C^{\infty}_{c}(S_\mathrm{tot}),
\]
so there exists $b\in H^{-m} \zeta^{-l} C^{\infty}_{c}(S_\mathrm{tot})$ for which
\[
b\sim\sum_\alpha \frac{i^{|\alpha|}}{\alpha!} (h\partial_\xi)^\alpha(\partial^\alpha_x a)(y,y,\xi;h).
\]

Fix any $N$.  Then 
\[
a - \sum_{|\alpha|<N} \frac{(x-y)^\alpha}{\alpha!}(\partial^{\alpha}_{x}a)(y,y,\xi;h) =: a_N\in C^\infty(\mathbb{T}^n;S^{m-N,l}(S_\mathrm{tot})),
\]
and
\[
b - \sum_{|\alpha|<N} \frac{i^{|\alpha|}}{\alpha!} (h\partial_\xi)^\alpha(\partial^\alpha_x a)(y,y,\xi;h) =: b_N\in S^{m-N,l}(S_\mathrm{tot}).
\]

We have
\begin{align*}
	I_h(a) &= \sum_{|\alpha|<N} \frac{1}{\alpha!} \int e^{\frac{i}{h}(x-y)\cdot\xi} \chi(x,y) (x-y)^\alpha (\partial^{\alpha}_{x}a)(y,y,\xi;h) \dbar\xi + I_h(a_N) \\
			&= \sum_{|\alpha|<N} \frac{1}{i^{|\alpha|}\alpha!} \int (h\partial_\xi)^\alpha e^{\frac{i}{h}(x-y)\cdot\xi} \chi(x,y) (\partial^{\alpha}_{x}a)(y,y,\xi;h) \dbar\xi + I_h(a_N) \\
			&= \sum_{|\alpha|<N} \frac{i^{|\alpha|}}{\alpha!} \int e^{\frac{i}{h}(x-y)\cdot\xi} \chi(x,y) (h\partial_\xi)^\alpha(\partial^{\alpha}_{x}a)(y,y,\xi;h) \dbar\xi +I_h(a_N) \\
			&= {}^h \mathrm{Op_r}(b - b_N) + I_h(a_N).
\end{align*}

Let $R_N := I_h(a_N) - {}^h \mathrm{Op_r}(b_N)\in \Psi^{m-N,l}_{2,h}(\mathcal{C})$.  We will show that the operator $I_h(a_N)$ is residual as defined above, `up to order $N$' (see below).  The proof for ${}^h \mathrm{Op_r}(b_N)$ is the same, except $b_N$ has one fewer component.


\textit{Proof of} Condition \ref{involutizing condition}: We are interested in $h^{-|\beta|} \mathbf{B}^\beta I_h(a_N)$.  By the Leibniz rule, this equals
\[
\sum_{\mu+\nu=\beta} \frac{\beta!}{\mu!\nu!} \int \left(\frac{1}{H}\right)^{\mu_1}\left(\frac{\Xi_2}{H}\right)^{\mu_2}\cdots\left(\frac{\Xi_d}{H}\right)^{\mu_d} e^{\frac{i}{h}(x-y)\cdot\xi}\left[\widetilde{\mathbf{B}}^\nu a_N\right] \dbar\xi.
\]
Recall that $H^{-1} = (\mathbf{v}_1\cdot\xi) / h$ is the reciprocal of the boundary defining function for the side face.  Hence, as long as $|\beta| \leq -m + N + l$, the net contribution is a positive power of $\rho_{\mathrm{sf}}$.
$I_h(a_N) \widetilde{\mathbf{B}}^\gamma$ is handled similarly.  Note that the amplitude $a_N$ is smooth in $x$ and $y$, so application of $\widetilde{\mathbf{B}}$ does not change the symbol class to which $a_N$ belongs.  By Proposition \ref{boundedness prop}, we conclude that $I_h(a_N)$ is involutizing for $|\beta| + m - l \leq N$.


Similarly, $I_h(a-a_N),{}^h \mathrm{Op_r}(b-b_N)\in\Psi^{m-N,l}_{2,h}(\mathcal{C})$ are finitely residual in this sense.  Then, since
\[
I_h(a) - {}^h \mathrm{Op_r}(b) = I_h(a-a_N) - {}^h \mathrm{Op_r}(b-b_N) + I_h(a_N) - {}^h \mathrm{Op_r}(b_N),
\]
we may conclude that $I_h(a) - {}^h \mathrm{Op_r}(b)$ is residual.
\color{black}
\end{proof}

We can likewise prove a \emph{left} reduction or Weyl reduction.  The point is that we can convert one quantization map into any other, at the (low) cost of one of these residual operators.




\begin{corollary}[Composition Law] \label{composition law}
If $a\in S^{m,l}(S_\mathrm{tot})$, $b\in S^{m',l'}(S_\mathrm{tot})$, then ${}^h \mathrm{Op_l}(a)\hspace{.1cm}\circ\hspace{.1cm}{}^h \mathrm{Op_l}(b) = {}^h \mathrm{Op_l}(c) + R$ for $c\in S^{m+m',l+l'}(S_\mathrm{tot})$ and $R\in\Re^{l+l'}$.
\end{corollary}

\begin{proof} By \autoref{composition theorem}, it is sufficient to consider ${}^h \mathrm{Op_l}(a)\circ {}^h \mathrm{Op_r}(b)$.  We have
\begin{align*}
& {}^h \mathrm{Op_l}(a) {}^h \mathrm{Op_r}(b) = \\
&= (2\pi h)^{-2n}\int{\int{\int{e^{\frac{i}{h}(x-w)\cdot\xi}e^{\frac{i}{h}(w-y)\cdot\eta} \chi(x,w)\chi(w,y) a(x,\xi;h)b(y,\eta;h)\hspace{.05cm}dw}\hspace{.05cm}d\eta}\hspace{.05cm}d\xi}.
\end{align*}
For this integral, we use the stationary phase theorem \cite[Theorem~3.17]{Zw}, and we concisely write the amplitude as $s(w,\eta;h)$.  By this theorem, for all $N$
\begin{align}
{}^h \mathrm{Op_l}(a) {}^h \mathrm{Op_r}(b) &= (2\pi h)^{-2n}\int{\int{\int{e^{\frac{i}{h}((x-y)\cdot\xi + (w-y)\cdot(\eta-\xi))} s(w,\eta;h)\hspace{.05cm}dw}\hspace{.05cm}d\eta}\hspace{.05cm}d\xi} \nonumber \\
&= (2\pi h)^{-n} \left(\int e^{\frac{i}{h}(x-y)\cdot\xi} \sum^{N-1}_{k=0} \frac{1}{i^k k!} \left[(D_w\cdot h D_\eta)^k s(w,\eta;h)\right]_{w=y,\eta=\xi}\ d\xi + O(h^N)\right) \label{stationary phase}.
\end{align}
Note that $(\partial_{w_j} \chi)(y,y) = 0$ and likewise for higher order derivatives of $\chi$, since these are supported off the diagonal.  Thus, writing out the first few terms of the sum,
\begin{align*}
& \sum^{N-1}_{k=0} \frac{1}{i^k k!} \left[(D_w\cdot h D_\eta)^k s(w,\eta;h)\right]_{w=y,\eta=\xi} = \\
&= \chi(x,y)a(x,\xi;h)b(y,\xi;h) + \sum^n_{j=1} (\partial_{w_j} \chi)(x,y)a(x,\xi;h)(h \partial_{\eta_j} b)(y,\xi;h) + O(h^2).
\end{align*}
We know $(\partial_{w_j} \chi)(x,y)$, and all higher derivatives evaluated at $w=y$, is supported off-diagonal.  So the contributions of these terms are residual (i.e., the kernel of a second microlocal operator is a coisotropic distribution, associated to $\mathcal{C}\times\mathcal{C}$, away from the diagonal).  Note that for all $j$, $(h \partial_{\eta_j} b)(y,\eta;h)$ is a second microlocal symbol, due to \eqref{hard lift}.

By \autoref{composition theorem}, we can then write \eqref{stationary phase} as the left quantization of some $c\in S^{m+m',l+l'}(S_\mathrm{tot})$, modulo a residual remainder.
\end{proof}

From now on, we will not explicitly write $\chi$.  By definition, the adjoint of a residual operator in $\Re^l$ is again an element of $\Re^l$.  In addition, it is routine to prove:

\begin{prop}
Let $A = {}^h \mathrm{Op_l}(a)$ for $a\in S^{m,l}(S_\mathrm{tot})$.  Then $A^* \in\Psi^{m,l}_{2,h}(\mathcal{C})$.
\end{prop}

Moreover, it is reassuring that quantizations of elements of $S^{-\infty,l}(S_\mathrm{tot})$, $l\in\mathbb{R}$, are residual.

\begin{prop} \label{quants are residual}
Suppose $\mathcal{C}$ is a linear coisotropic.  If $A = {}^h \mathrm{Op_r}(a)$ for $a\in S^{-\infty,l}(S_\mathrm{tot})$, then $A\in\Re^l$.
\end{prop}

\begin{proof} Locally in $S_\mathrm{tot}$, we may take $\mathbf{v}_1\cdot\xi$ as defining function for the front face and $h / {(\mathbf{v}_1\cdot\xi)}$ for side face defining function.

We show that $A$ satisfies Condition \ref{involutizing condition}.  Recall that $\widetilde{B}_j = \mathbf{v}_j \cdot D_x$ and $\widetilde{\mathbf{B}}^\beta = \widetilde{B}_1^{\beta_1}\circ\cdots\circ\widetilde{B}_d^{\beta_d}$; also, let $\mathbf{V} = (\mathbf{v}_1\cdot\xi,\ldots,\mathbf{v}_d\cdot\xi)$.  We have
\begin{equation*}
\widetilde{\mathbf{B}}^\beta A = \int \left(\frac{\mathbf{V}}{h}\right)^\beta e^{\frac{i}{h}(x-y)\cdot\xi} a(y,\xi;h) \ \dbar\xi = \int \rho_\mathrm{sf}^{-|\beta|} \Xi^{(\beta_2,\ldots,\beta_d)} e^{\frac{i}{h}(x-y)\cdot\xi} a(y,\xi;h) \ \dbar\xi.
\end{equation*}
No matter how large $\beta$ is, $\rho_\mathrm{sf}^{-|\beta|} a\in S^{0,l}(S_\mathrm{tot})$.  Therefore, $\widetilde{\mathbf{B}}^\beta A$ maps $L^2(\mathbb{T}^n)$ to $h^{-l} L^2(\mathbb{T}^n)$.
\end{proof}

Finally, we show:

\begin{prop} \label{two-sided ideal prop}
The residual algebra is an ideal in $\Psi_{2,h}(\mathcal{C})$.
\end{prop}

\begin{proof}
Let $a\in S^{m,l}(S_\mathrm{tot})$ and $A = {}^h \mathrm{Op}_r(a)$.  Let $R\in\Re^{l'}$.  We show that $AR\in\Re^{l+l'}$.  Taking adjoints then implies $RA\in\Re^{l+l'}$.

Let $\Theta_{x_j} = \mathbf{v}_j\cdot D_x$ and $\Theta_x^\beta = \Theta_{x_1}^{\beta_1}\cdots\Theta_{x_d}^{\beta_d}$.  Then, for $u\in L^2(\mathbb{T}^n)$,
\begin{align*}
h^{l+l'} \Theta_x^\beta A R u(x) &= (-1)^{|\beta|} h^{l+l'} \int\int \left(\frac{\mathbf{V}}{h}\right)^\beta \left(\frac{h}{\mathbf{V}} \Theta_y\right)^\beta e^{\frac{i}{h}(x-y)\cdot\xi} a(y,\xi;h) Ru(y)\ dy \dbar\xi \\
&= \sum_{\mu+\nu = \beta} \frac{\beta!}{\mu!\nu!} h^l \int\int e^{\frac{i}{h}(x-y)\cdot\xi} \Theta_y^\mu a(y,\xi;h) \left[h^{l'} \Theta_y^\nu Ru(y)\right] dy \dbar\xi.
\end{align*}
The bracketed term lies in $L^2$, since $R$ is involutizing.  And $h^l\ {}^h \mathrm{Op}_r(\Theta_y^\mu a)\in\Psi^{m-l,0}_{2,h}(\mathcal{C})$ so, if $m\leq l$, it is $L^2$ bounded.  Instead, if $m>l$, choose $p\geq m-l$.  As before, locally we may take $\rho_\mathrm{sf} = h/{\mathbf{v}_1\cdot\xi}$ as side face defining function.  Then for $\mu+\nu = \beta$,
\[
(-1)^p h^l \int\int \left(\frac{h}{\mathbf{v}_1\cdot\xi} \Theta_{y_1}\right)^p e^{\frac{i}{h}(x-y)\cdot\xi} \Theta_y^\mu a(y,\xi;h) \left[h^{l'} \Theta_y^\nu Ru(y)\right] dy \dbar\xi
\]
\[
= \sum_{i+j = p} \frac{p!}{i!j!} h^l \int\int \rho_\mathrm{sf}^p e^{\frac{i}{h}(x-y)\cdot\xi} \Theta_{y_1}^i \Theta_y^\mu a(y,\xi;h) \left[h^{l'} \Theta_{y_1}^j \Theta_y^\nu Ru(y)\right] dy \dbar\xi.
\]
This time, the amplitude belongs to $S^{m-l-p,0}(S_\mathrm{tot})$, so by Proposition \ref{boundedness prop} the operator is $L^2$ bounded.  Thus, $AR$ satisfies Condition \ref{involutizing condition}.

Meanwhile, it is easily seen that for $R_1\in\Re^l$, $R_2\in\Re^{l'}$, we have $R_1 R_2\in\Re^{l+l'}$.
\end{proof}

\section{\textsc{Microsupport, parametrices \& second wavefront}\label{sec:second principal}}

\subsection{Second microsupport}

Let $S_\mathrm{tot}$ and $S_\mathrm{pr}$ be symbol spaces associated to any linear coisotropic $\mathcal{C}$.  Recall that $S_\mathrm{pr}$ may be identified with the side face of $S_\mathrm{tot}$.

\begin{definition}
For $a\in S^{m,l}(S_\mathrm{tot})$, define the \emph{essential support} of $a$ by:
\[
S_\mathrm{pr}\backslash\mathrm{ess \ supp}_l(a) := \{p\in S_\mathrm{pr} \ | \ \exists \varphi\in C^\infty_c (S_\mathrm{tot}), \varphi(p)\neq 0, \varphi a\in S^{-\infty,l}(S_\mathrm{tot})\}.
\]
If $A = {}^h \mathrm{Op_r}(a)$ for $a\in S^{m,l}(S_\mathrm{tot})$, then the \emph{second microsupport} ${}^2 \mathrm{WF}_l'(A)$ of $A$ is given by ${}^2 \mathrm{WF}_l'(A) := \mathrm{ess \ supp}_l(a)$.  If $A \in \Re^l$, define ${}^2 \mathrm{WF}_l'(A) = \emptyset$.
\end{definition}

Second microsupport obeys the usual laws of microsupports:

\begin{prop} \label{laws of microsupports}
Let $A,B\in\Psi^{m,l}_{2,h}(\mathcal{C})$, $D\in\Psi^{m',l'}_{2,h}(\mathcal{C})$.  Then ${}^2 \mathrm{WF}'$ satisfies $${}^2 \mathrm{WF}_l'(A+B) \subset {}^2 \mathrm{WF}_l'(A) \cup {}^2 \mathrm{WF}_l'(B), \ \ {}^2 \mathrm{WF}_{l+l'}'(AD) \subset {}^2 \mathrm{WF}_l'(A) \cap {}^2 \mathrm{WF}_{l'}'(D).$$
\end{prop}

\subsection{Principal symbols of second microlocal operators}

Let $A\in\Psi^{m,l}_{2,h}(\mathcal{C})$.  Suppose $A = {}^h \mathrm{Op_r}(a) + R$ for some $a\in S^{m,l}(S_\mathrm{tot})$ and residual operator $R\in\Re^l$.

\begin{definition}
The \emph{principal symbol} of $A$ is
\[
{}^2 \sigma_{m,l}(A) := (h^m a)\restriction_\mathrm{sf} \ \in S^{l-m}(S_\mathrm{pr}).
\]
\end{definition}

But the pair $(a,R)$ is not unique.  Suppose we also have $a'\in S^{m,l}(S_\mathrm{tot})$, $R'\in\Re^l$ such that $A = {}^h \mathrm{Op_r}(a') + R'$.  Then
\[
{}^h \mathrm{Op_r}(a - a') = R' - R \in \Re^l.
\]
Thus, while \emph{a priori} ${}^h \mathrm{Op_r}(a - a')\in\Psi^{m,l}_{2,h}(\mathcal{C})$, in fact ${}^h \mathrm{Op_r}(a - a')\in\Re^l$.  This principal symbol will be well defined if it is independent of the choice of right reduction.  This amounts to showing that the difference $a-a'$, which \emph{a priori} belongs to $S^{m,l}(S_\mathrm{tot})$ for whatever value of $m\in\mathbb{R}$, in fact decays at the side face of $S_\mathrm{tot}$.  We claim the stronger result that $a-a'\in S^{-\infty,l}(S_\mathrm{tot})$.




\begin{lemma} \label{lem:more modest}
For $b\in S^{m,l}(S_\mathrm{tot})$, suppose that ${}^h \mathrm{Op_r}(b)\in\Re^l$.  Then $b\in S^{-\infty,l}(S_\mathrm{tot})$.
\end{lemma}

\begin{proof}
Let $\mathcal{C} = \{\mathbf{v}_1\cdot\xi = \ldots = \mathbf{v}_d\cdot\xi = 0\}$ as usual.  We use the same coordinates in $S_\mathrm{tot}$ as in the proof of Proposition \ref{boundedness prop}: $\rho_\mathrm{ff} = \mathbf{v}_1\cdot\xi$, $\Xi$, and $\mathbf{W}$.  (Note that $\rho_\mathrm{sf} = h / {(\mathbf{v}_1\cdot\xi)}$, as we must have $\rho_\mathrm{ff}\times\rho_\mathrm{sf} = h$.)

Expand $b$ in powers of $\rho_\mathrm{sf}$, near the corner $\{\rho_\mathrm{sf} = \rho_\mathrm{ff} = 0\}$:
\[
b(x,\rho_\mathrm{sf},\rho_\mathrm{ff},\Xi,\mathbf{W}) \sim \sum_j \rho_\mathrm{sf}^j b_j(x,\rho_\mathrm{ff},\Xi,\mathbf{W});
\]
the coefficients $b_j$ can be taken to be smooth functions on the side face.

Next, suppose $b\neq O(\rho_\mathrm{sf}^\infty)$.  Then there must exist a coefficient $b_{j_0}$ which is nontrivial at the corner.  By continuity of $b_{j_0}$, we have $b_{j_0}$ nontrivial in a neighborhood $\{\rho_\mathrm{ff} < \epsilon\}$ (also localized in $x$, $\Xi$, and $\mathbf{W}$); i.e., there exists a point $(x^0,\rho_\mathrm{ff}^0 > 0,\Xi^0,\mathbf{W}^0)$ at which $b_{j_0}$ is nonzero.  So $\rho_\mathrm{sf}^{j_0} b_{j_0}\neq O(h^\infty)$ at $(x^0,\rho_\mathrm{sf} = 0,\rho_\mathrm{ff}^0,\Xi^0,\mathbf{W}^0)$.  Hence ${}^h \mathrm{Op_r}(b)$ cannot lie in $\Re^l$, since $\Re^l$ reduces to $O(h^\infty)$ along the side face.  (More generally, $\Psi_{2,h}(\mathcal{C})$ is really just $\widetilde{\Psi}_h(\mathbb{T}^n)$ microlocally away from the coisotropic.)
\end{proof}


Set ${}^2 \mathrm{char}_{m,l}(A) := \{\rho_\mathrm{ff}^{l-m} \ {}^2 \sigma_{m,l}(A) = 0\}$.  That is, the second characteristic set is the zero set of a smooth function on the principal symbol space.  We abuse notation by omitting the indices and writing ${}^2 \mathrm{char}_{m,l}$ simply as ${}^2 \mathrm{char}$.  The complementary notion is
\[
{}^2 \mathrm{ell}(A) = {}^2 \mathrm{ell}_{m,l}(A) := \{\rho_\mathrm{ff}^{l-m} \ {}^2 \sigma_{m,l}(A)\neq 0\}.
\]

Note that these definitions are independent of choice of front face defining function.  Next, we state some essential properties of the principal symbol map.

\begin{lemma} \label{short exact sequence}
\[
0\longrightarrow\Psi^{m-1,l}_{2,h}(\mathcal{C})\longrightarrow\Psi^{m,l}_{2,h}(\mathcal{C})\xrightarrow{{}^2 \sigma_{m,l}}S^{l-m}(S_\mathrm{pr})\longrightarrow 0
\]
is a short exact sequence.  Furthermore, the principal symbol map is a homomorphism: if $A\in\Psi^{m,l}_{2,h}(\mathcal{C})$, $B\in\Psi^{m',l'}_{2,h}(\mathcal{C})$, then ${}^2 \sigma_{m+m',l+l'}(AB)={}^2 \sigma_{m,l}(A) {}^2 \sigma_{m',l'}(B)$.
\end{lemma}

The proof is straightforward, so we omit it.

\begin{remark}
If $A\in\Psi^{m,l}_{2,h}(\mathcal{C})$, $B\in\Psi^{m',l'}_{2,h}(\mathcal{C})$, then
\[
{}^2 \sigma_{m+m'-1,l+l'}(i[A,B]) = \{{}^2 \sigma_{m,l}(A),{}^2 \sigma_{m',l'}(B)\},
\]
where the Poisson bracket is computed with respect to the symplectic form on $S_\mathrm{pr}$ lifted from the symplectic form on $T^* \mathbb{T}^n$.  See also Remark \ref{rmk:scaling}.
\end{remark}

\begin{remark} \label{one calculus in the other}
For each $m\in\mathbb{R}$, $\widetilde{\Psi}^{m}_{h}(\mathbb{T}^n)$ can be identified with a subset of $\Psi^{m,m}_{2,h}(\mathcal{C})$.  Locally, each element $A\in\widetilde{\Psi}^{m}_{h}(\mathbb{T}^n)$ is a quantization of a symbol $a\in h^{-m} C^\infty_c (T^* \mathbb{T}^n \times [0,1)_h)$ (or $a$ is residual in both semiclassical and differential filtrations).  If $\beta_\mathrm{tot}: S_\mathrm{tot} \rightarrow T^* \mathbb{T}^n \times [0,1)$ is the (smooth) blowdown map, then the pullback $\beta_\mathrm{tot}^* a$ quantizes to the corresponding element $A\in\Psi^{m,m}_{2,h}(\mathcal{C})$.  Likewise, we may identify the principal symbol of $A\in\widetilde{\Psi}^{m}_{h}(\mathbb{T}^n)$ with the second principal symbol of $A\in\Psi^{m,m}_{2,h}(\mathcal{C})$.
\end{remark}

\subsection{Global Parametrices}


\begin{lemma} \label{global parametrix}
Suppose that $A\in\Psi^{m,l}_{2,h}(\mathcal{C})$ is globally elliptic (i.e., ${}^2 \sigma_{m,l}(A)$ vanishes nowhere).  Then there exist $B,C\in\Psi^{-m,-l}_{2,h}(\mathcal{C})$ for which $AB-\mathrm{Id}\in\Re^0$, $CA-\mathrm{Id}\in\Re^0$; and $B$ and $C$ differ by an element of $\Re^{-l}$. 
\end{lemma}

\begin{proof}
An iterative construction similar to the proof of \cite[Theorem~3.4]{Wu13}.
\end{proof}

\subsection{Microlocal Parametrices}

Let $\mathcal{C}=\{\mathbf{v}_1\cdot\xi = \ldots = \mathbf{v}_d\cdot\xi = 0\}$ for $\mathbf{v}_i\in\mathbb{R}^n$.  Let $A\in\Psi^{m,l}_{2,h}(\mathcal{C})$.  We partition the principal symbol space $S_\mathrm{pr}$, and restrict our focus to the lift $L_1$ of $\{\mathbf{v}_1\cdot\xi\neq 0\}\subset T^* \mathbb{T}^n$ (we have likewise partitioned $S_\mathrm{tot}$).  Extend $\{\mathbf{v}_1,\ldots,\mathbf{v}_d\}$ to a basis
\[
\{\mathbf{v}_1,\ldots,\mathbf{v}_d,\mathbf{w}_{d+1},\ldots,\mathbf{w}_n\}
\]
for $\mathbb{R}^n$.  Near the corner $SN(\mathcal{C}) = \partial S_\mathrm{pr}$, valid coordinates are $x$, $\zeta = \mathbf{v}_1\cdot\xi$, $$\Xi:=\left(\frac{\mathbf{v}_2\cdot\xi}{\mathbf{v}_1\cdot\xi},\ldots,\frac{\mathbf{v}_d\cdot\xi}{\mathbf{v}_1\cdot\xi}\right), \ \mathrm{and} \ \mathbf{W}:=(\mathbf{w}_{d+1}\cdot\xi,\ldots,\mathbf{w}_n\cdot\xi);$$ $h=0$ on the entirety of $S_\mathrm{pr}$.

\begin{lemma} \label{microlocal parametrix}
Take a point $\overline{p} = \left(\overline{x},\overline{\zeta},\overline{\Xi},\overline{\mathbf{W}}\right)\in {}^2 \mathrm{ell}(A) \cap L_1$.  Then there exists $B\in\Psi^{-m,-l}_{2,h}(\mathcal{C})$ such that
\[
\overline{p}\notin{}^2 \mathrm{WF}_0'(AB-\mathrm{Id}) \cup {}^2 \mathrm{WF}_0'(BA-\mathrm{Id}).
\]
\end{lemma}

\begin{proof}
Similar to the construction of microlocal parametrix presented in Lemma 4.3 of \cite{Me-2}, but adapted to the second microlocal setting.

\end{proof}

A neat consequence of microlocal parametrices is the following elliptic regularity result:


\begin{corollary} \label{elliptic regularity}
Suppose $P\in\Psi^{m,l}_{2,h}(\mathcal{C})$ is elliptic on the microsupport of $A\in\Psi^{m',l'}_{2,h}(\mathcal{C})$:
\[
{}^2 \mathrm{WF}_{l'}'(A) \subset {}^2 \mathrm{ell}(P) \subset S_\mathrm{pr}.
\]
Then there exists a second microlocal operator $A_0\in\Psi^{-m+m',-l+l'}_{2,h}(\mathcal{C})$ such that $A = A_0 P + \Re^{l'}$.
\end{corollary}



\subsection{Mapping Properties}

As usual, let $\mathcal{C} = \{\mathbf{v}_1\cdot\xi = \ldots = \mathbf{v}_d\cdot\xi = 0\}$.  Next, for $k,s\in\mathbb{R}$, we give an alternative characterization of $I^k_{(s)} (\mathcal{C})$ (these spaces of distributions were defined in \eqref{eq:distributions}):

\begin{lemma} \label{alternative characterization lemma}
\begin{equation} \label{alternative characterization}
I^k_{(s)} (\mathcal{C}) = \{u\in h^s L^2(\mathbb{T}^n) \ | \ \exists \ \mathrm{globally \ elliptic} \ A\in\Psi^{k,0}_{2,h}(\mathcal{C}), \ Au\in h^s L^2(\mathbb{T}^n)\}
\end{equation}
\end{lemma}



\begin{proof}
For simplicity, we may as well assume $s = 0$.  We prove the lemma directly for $k\in\mathbb{Z}_{\geq 0}$; interpolation and duality then give the full result.

Assume $u\in I^k_{(0)} (\mathcal{C})$.  For $x\in\mathbb{T}^n$, $(x,\xi)\in\mathcal{C}$ if and only if $(\mathbf{v}_1\cdot\xi)^2 + \ldots + (\mathbf{v}_d\cdot\xi)^2 = 0$.  So
\[
B := h^{2} ((\mathbf{v}_1\cdot D_x)^2 + \ldots + (\mathbf{v}_d\cdot D_x)^2)
\]
belongs in $\mathfrak{M}_\mathcal{C}$.  Hence, $h^{-1} B\in\Psi^1_h (\mathbb{T}^n)$ and $(h^{-1} B)^k u\in L^2(\mathbb{T}^n)$.  At the same time, we take $\rho_\mathrm{ff} = |(\mathbf{v}_1\cdot\xi,\ldots,\mathbf{v}_d\cdot\xi)|$ as front face defining function.  The symbol of $(h^{-1} B)^k$ is $h^{-k}|(\mathbf{v}_1\cdot\xi,\ldots,\mathbf{v}_d\cdot\xi)|^{2k}$, which lifts to $\rho_\mathrm{sf}^{-k} \rho_\mathrm{ff}^k$, so $(h^{-1} B)^k\in\Psi^{k,-k}_{2,h}(\mathcal{C})\subset\Psi^{k,0}_{2,h}(\mathcal{C})$\footnote{Really we should first `cut off' $B$ so its symbol is compactly supported in the fibers.  Then only can it be regarded as a second microlocal operator.  Similarly for the $B_j$ in the subsequent paragraph.} (since $k\geq 0$).  Then $A = (I + h^{-1} B)^k$
is an elliptic operator satisfying $Au\in L^2(\mathbb{T}^n)$.


Conversely, consider the generators $B_1,\ldots,B_d$, $B_j = \mathbf{v}_j \cdot hD_x$, of $\mathfrak{M}_\mathcal{C}$.  Suppose there exists elliptic $A\in\Psi^{k,0}_{2,h}(\mathcal{C})$ such that $Au\in L^2(\mathbb{T}^n)$.  Let $0\leq |\beta|\leq k$.  Then there exists a parametrix $A'\in\Psi^{-k,0}_{2,h}(\mathcal{C})$ such that
\[
h^{-|\beta|} \mathbf{B}^\beta u = h^{-|\beta|} \mathbf{B}^\beta A' A u + h^{-|\beta|} \mathbf{B}^\beta R u, \ \ R\in\Re^0.
\]
Note that $h^{-|\beta|} \mathbf{B}^\beta\in\Psi^{|\beta|,0}_{2,h}(\mathcal{C})$ (for each $1\leq j\leq d$, $\mathbf{v}_j\cdot\xi$ is a locally valid defining function for the front face).  By Proposition \ref{two-sided ideal prop}, $h^{-|\beta|} \mathbf{B}^\beta R\in\Re^0$, so the latter summand in the RHS above lies in $I^\infty_{(0)}(\mathcal{C})\subset L^2(\mathbb{T}^n)$.  Since $|\beta| \leq k$, $h^{-|\beta|} \mathbf{B}^\beta A'\in\Psi^{0,0}_{2,h}(\mathcal{C})$, so Proposition \ref{boundedness prop} implies that $h^{-|\beta|} \mathbf{B}^\beta A' (A u)\in L^2(\mathbb{T}^n)$.
\end{proof}

\begin{lemma}[Mapping Property] \label{mapping property lemma}
For $m\in\mathbb{R}$ and $l\geq 0$, $P\in\Psi^{m,l}_{2,h}(\mathcal{C})$ satisfies
\begin{equation} \label{mapping property}
P:I^k_{(s)}(\mathcal{C})\longrightarrow I^{k-m}_{(s-l)}(\mathcal{C})
\end{equation}
for each $k,s\in\mathbb{R}$.  In particular, if $R\in\Re^l$, then
\[
R:I^{-\infty}_{(s)}(\mathcal{C})\longrightarrow I^\infty_{(s-l)}(\mathcal{C}).
\]
\end{lemma}

Since $I^0_{(0)}(\mathcal{C}) = L^2(\mathbb{T}^n)$, this property generalizes Proposition \ref{boundedness prop}.  Define
\begin{equation} \label{partial Sobolev norm}
\|u_h\|_{I^k_{(s)}} := \left\|\mathcal{F}^{-1}_h \left(1 + |(\mathbf{v}_1\cdot\xi,\ldots,\mathbf{v}_d\cdot\xi)|^2\right)^{k/2} \mathcal{F}_h(h^{-s} u_h)\right\|_{L^2(\mathbb{T}^n)}.
\end{equation}
$\mathcal{F}_h$ is the semiclassical Fourier transform.  We see that the $I^k_{(s)}$-norm is a partial Sobolev norm with respect to the characteristic derivatives.

\begin{proof}

Take $u\in I^k_{(s)}(\mathcal{C})$.  By Lemma \ref{alternative characterization lemma}, there is an elliptic operator $A\in\Psi^{k,0}_{2,h}(\mathcal{C})$ for which $Au\in h^s L^2(\mathbb{T}^n)$.  Choose $P\in\Psi_{2,h}^{m,l}(\mathcal{C})$.  We want to prove there exists elliptic $\widetilde{A}\in\Psi^{k-m,0}_{2,h}(\mathcal{C})$ satisfying $\widetilde{A}Pu\in h^{s-l} L^2(\mathbb{T}^n)$.  Let $\widetilde{A}$ be any elliptic element of $\Psi^{k-m,0}_{2,h}(\mathcal{C})$.  At the same time, since $A$ is elliptic, there exists $B\in\Psi^{-k,0}_{2,h}(\mathcal{C})$ such that $BA - \mathrm{Id} = R\in\Re^0$.  Therefore,
\[
\widetilde{A}Pu = \widetilde{A}P(BA - R)u = \widetilde{A}PB(Au) - (\widetilde{A}PR)u.
\]
$\widetilde{A}PR\in\Re^l$ (by Proposition \ref{two-sided ideal prop}) and $\widetilde{A}PB\in\Psi^{0,l}_{2,h}(\mathcal{C})$.  Since $u\in h^s L^2(\mathbb{T}^n)$, certainly we have $(\widetilde{A}PR)u\in h^{s-l} L^2(\mathbb{T}^n)$.  And since $l\geq 0$, $h^l \widetilde{A}PB\in\Psi^{-l,0}_{2,h}(\mathcal{C})$ maps $h^s L^2$ into itself, by Proposition \ref{boundedness prop}.
\end{proof}

\subsection{Second Wavefront Set}

\begin{definition} \label{def of 2-wavefront}
For any $m,l\in\mathbb{R}$, the $m,l$-graded \emph{second wavefront set} of a distribution $u\in I^{-\infty}_{(l)}(\mathcal{C})$ is
\[
{}^2 \mathrm{WF}^{m,l}(u) = \bigcap\{{}^2 \mathrm{char}(A) \ | \ A\in\Psi^{m,l}_{2,h}(\mathcal{C}),Au\in L^2(\mathbb{T}^n)\}.
\]
\end{definition}

Let $u\in I^{-\infty}_{(l)}(\mathcal{C})$.  We will later use the containment property
\begin{equation} \label{eqn:containment}
m \leq m' \ \Longrightarrow \ {}^2 \mathrm{WF}^{m,l}(u) \subset {}^2 \mathrm{WF}^{m',l}(u).
\end{equation}
Also, put ${}^2 \mathrm{WF}^{\infty,l}(u) = \overline{\bigcup_{m\in\mathbb{R}} {}^2 \mathrm{WF}^{m,l}(u)}$, so that
\[
{}^2 \mathrm{WF}^{\infty,l}(u) = \bigcap\{{}^2 \mathrm{char}(A) \ | \ A\in\Psi^{m,l}_{2,h}(\mathcal{C}),Au\in I^\infty_{(0)}(\mathcal{C})\}.
\]
Notice that ${}^2 \mathrm{WF}^{\infty,l}(u) = \emptyset$ if and only if $u\in I^\infty_{(l)}(\mathcal{C})$, i.e., $u$ is a coisotropic distribution.

We see that ${}^2 \mathrm{WF}^{m,l}(u)\subset S_\mathrm{pr}$.  Away from the coisotropic, this new wavefront is the same as standard semiclassical wavefront set:
\[
{}^2 \mathrm{WF}^{m,l}(u) \backslash SN(\mathcal{C}) \simeq \mathrm{WF}_h^m(u) \backslash \mathcal{C}.
\]
Just as the semiclassical pseudodifferential calculus $\Psi_h$ does not spread singularities as measured by $\mathrm{WF}_h$, we have:

\begin{lemma} \label{lem:pseudolocality}
Let $A\in\Psi^{m',l'}_{2,h}(\mathcal{C})$.  For $u\in I^{-\infty}_{(l)}(\mathcal{C})$, we have
\[
{}^2 \mathrm{WF}^{m-m',l-l'}(Au) \subset {}^2 \mathrm{WF}_{l'}'(A) \cap {}^2 \mathrm{WF}^{m,l}(u).
\]
\end{lemma}

\begin{proof}
Analogous to the proof of microlocality in \cite[Proposition~4.2]{Me-2}.
\end{proof}

As a consequence, there is a partial converse to Lemma \ref{lem:pseudolocality}.

\begin{corollary} \label{wavefront is characteristic}
Let $A\in\Psi^{m',l'}_{2,h}(\mathcal{C})$ and $u\in I^{-\infty}_{(l)}(\mathcal{C})$.  Then for any $m\in\mathbb{R}$,
\[
{}^2 \mathrm{WF}^{m,l}(u) \subset {}^2 \mathrm{WF}^{m-m',l-l'}(Au) \cup {}^2 \mathrm{char}(A).
\]
\end{corollary}

\begin{proof}
Similar to the proof of \cite[Proposition~4.3]{Me-2}.
\end{proof}


\begin{example} \label{ex:second wavefront}
Let $k\in\mathbb{Z}$.  In $\mathbb{T}^n$, we consider the distribution
\[
u_k = \exp[i(k(x_1 + \ldots + x_{n-1}) + k^2 x_n)].
\]
Letting
\[
h^2_k = \frac{1}{(n-1)k^2 + k^4} \xrightarrow{|k|\rightarrow\infty} 0,
\]
we see that $u_k$ satisfies $(h^2_k \Delta - 1)u_k = 0$.

At the same time, we find that
\begin{align}
\mathrm{WF}_{h_k}(u_k) &= \{\xi_1=\ldots=\xi_{n-1}=0, \xi_n = 1\} \label{fiber} \\
&\subset \{\xi_1=\ldots=\xi_{n-1}=0, \xi_n\in\mathbb{R}\} =: \mathcal{C} \nonumber
\end{align}
Therefore, whatever second wavefront there is will lie in $SN(\mathcal{C})$.

For $\mathcal{C}$ as defined above, $\mathfrak{M}_\mathcal{C}$ is generated by $hD_{x_j}$, $j < n$.  Fix $j$.  Then we have
\[
h^{-1}(h D_{x_j}) u_k = D_{x_j} u_k = k u_k,
\]
so $D_{x_j} u_k\notin L^2$ uniformly as $|k|\rightarrow\infty$.  Thus, $u_k$ is not a coisotropic distribution, so $\emptyset \neq {}^2 \mathrm{WF}^{\infty,0}(u_k)\subset SN(\mathcal{C})$.

Let $\xi' = (\xi_1,\ldots,\xi_{n-1})$ and $\hat{\xi}' = \xi' / |\xi'|$.  Due to \eqref{fiber}, ${}^2 \mathrm{WF}^{\infty,0}(u_k) \subset \{|\hat{\xi}'| = 1, \xi_n = 1\}$.

Next, consider the operators $h(D_{x_i} - D_{x_j})\in\Psi^0_h(\mathbb{T}^n)$, $1\leq i,j\leq (n-1)$, formed from the generators of $\mathfrak{M}_\mathcal{C}$.  Unlike the generators themselves, $h^{-1}(hD_{x_i} - hD_{x_j}) u_k = 0$.  We may regard $D_{x_i} - D_{x_j}$ as a second microlocal operator\footnote{Again, after truncation.}.
Since $(D_{x_i} - D_{x_j}) u_k = 0$,
\[
{}^2 \mathrm{WF}^{\infty,0}(u_k) \subset \left\{{}^2 \sigma(D_{x_i} - D_{x_j}) = \frac{\xi_i}{h} - \frac{\xi_j}{h} = 0\right\}.
\]
So we must have $\hat{\xi}_i = \hat{\xi}_j$ for these $i,j$.  This leads to the following claim:

\emph{Claim:} If $k\rightarrow +\infty$, then
\[
{}^2 \mathrm{WF}^{\infty,0}(u_k) = \left\{\hat{\xi}_1 = \ldots = \hat{\xi}_{n-1} = \frac{1}{\sqrt{n-1}}, \xi_n = 1\right\}.
\]
If instead $k\rightarrow -\infty$,
\[
{}^2 \mathrm{WF}^{\infty,0}(u_k) = \left\{\hat{\xi}_1 = \ldots = \hat{\xi}_{n-1} = -\frac{1}{\sqrt{n-1}}, \xi_n = 1\right\}.
\]

\emph{Proof of claim:} In both cases, all of the $\hat{\xi}_j$ ($j < n$) have the same sign.  So it suffices to construct a second microlocal operator $A$ that `distinguishes' signs for, say, $\hat{\xi}_1$, and also for which $A u_k\in I^\infty_{(0)}(\mathcal{C})$.

We partition $S_\mathrm{tot}$ into $(n-1)$ symmetric pieces, the blowup of $\{\xi_j\neq 0\}\times\{h\geq 0\}$ for each of $j = 1,\ldots,n-1$.  Since WLOG we chose to distinguish $\hat{\xi}_1>0$ from $\hat{\xi}_1<0$, we write a symbol for $A$ which is supported in the lift of $\{\xi_1\neq 0\}\times\{h\geq 0\}$ (and which is zero on the other $(n-2)$ pieces).

Now, suppose $k\rightarrow +\infty$.  We write down $a\in C^\infty_c (S_\mathrm{tot})$ so that $A = {}^h \mathrm{Op_r}(a)$ is elliptic for $\hat{\xi}_1<0$ (but characteristic for $\hat{\xi}_1>0$), and so that $A u_k\in I^\infty_{(0)}(\mathcal{C})$.  This will exclude the point
\[
\left(x,\hat{\xi}_1=-\frac{1}{\sqrt{n-1}},\ldots,\hat{\xi}_{n-1}=-\frac{1}{\sqrt{n-1}},\xi_n=1\right)
\]
from ${}^2 \mathrm{WF}^{\infty,0}(u_k)$.

Let $\psi\in C^\infty(\mathbb{R})$ be supported in $(0,\infty)$.  Let $\varphi_j(x_j)\in C^\infty_c(S^1)$ be a bump function supported near some (arbitrary) $\bar{x}_j$, and let $\varphi(x) = \prod^n_{j=1} \varphi_j(x_j)$.  Define
\[
a(x,\xi;h) := \psi\left(-\frac{\xi_1}{h}\right) \varphi(x)\in C^\infty_c(S_\mathrm{tot}).
\]

Since $A = {}^h \mathrm{Op_r}(a)$, we have
\[
A u_k(x) = \frac{1}{(2\pi h)^n} \int_{\mathbb{R}^n} \int_{\mathbb{T}^n} e^{\frac{i}{h}(x-y)\cdot\xi} \psi\left(-\frac{\xi_1}{h}\right) \varphi(y) u_k(y)\ dyd\xi.
\]
Let $\mathcal{F}$ be the \emph{non-semiclassical} Fourier transform; recall this takes smooth, compactly supported functions to Schwartz functions.  We change variables by setting $\eta_j = \xi_j / h$, so that
\[
A u_k(x) = \mathcal{F}^{-1}_{\eta\rightarrow x} (\psi(-\eta_1) \mathcal{F}_{y\rightarrow\eta} \varphi u_k)(x).
\]

We compute
\begin{align*}
\mathcal{F}_{y\rightarrow\eta} (\varphi u_k)(\eta) &= \int_{\mathbb{T}^n} e^{-iy\cdot\eta} \varphi(y) e^{i(k(y_1+\ldots+y_{n-1})+k^2 y_n)}\ dy \\
&= \left[\prod^{n-1}_{j=1} \int_{S^1} e^{-iy_j\eta_j} \varphi_j(y_j) e^{iky_j}\ dy_j\right]\left[\int_{S^1} e^{-iy_n\eta_n} \varphi_n(y_n) e^{ik^2 y_n}\ dy_n\right].
\end{align*}
Hence,
\[
A u_k(x) = \mathcal{F}^{-1}_{\eta\rightarrow x} (\psi(-\eta_1) \widehat{\varphi}_1(\eta_1-k)\cdots\widehat{\varphi}_{n-1}(\eta_{n-1}-k)\widehat{\varphi}_n(\eta_n-k^2))(x).
\]

Finally, since $\psi(-\eta_1) = \psi(-\xi_1/h)$ is supported in $\{\eta_1\leq 0\}$, yet $\widehat{\varphi}_1(\eta_1-k)$ is a Schwartz function whose supported is translated to the right as $k\rightarrow +\infty$, we may conclude that $A u_k = O_{L^2}(h_k^\infty)$.  This is stronger than $A u_k\in I^\infty_{(0)}(\mathcal{C})$, so certainly
\[
{}^2 \mathrm{WF}^{\infty,0}(u_k) = \mathbb{T}^n \times \left\{\hat{\xi}_1=\ldots=\hat{\xi}_{n-1}=\frac{1}{\sqrt{n-1}}, \xi_n=1\right\}.
\]
\end{example}

\section{\textsc{Results on propagation of second wavefront} \label{propagation results}}

\subsection{Real principal type propagation}

Let $\mathcal{C}\subset T^* \mathbb{T}^n$ be any linear coisotropic.

\begin{lemma} \label{lem:bounding}
For $s,r,\alpha,\beta\in\mathbb{R}$, let $f\in I^\alpha_{(r)} (\mathcal{C})$, $g\in I^\beta_{(s)} (\mathcal{C})$, and $P\in\Re^{r+s}$.  Then
\[
|\left\langle Pf,g \right\rangle_{L^2(\mathbb{T}^n)}| \lesssim \|f\|_{I^\alpha_{(r)}} \|g\|_{I^\beta_{(s)}}
\]
\end{lemma}

For the definition of the norm $\|\cdot\|_{I^\alpha_{(r)}}$, refer to \eqref{partial Sobolev norm}.

\begin{proof}
Essentially, we factor $P\in\Re^{r+s}$ as the product of a residual operator of order $r$ and a residual operator of order $s$.  For all $N$, we may choose globally elliptic $T\in\Psi^{-N,s}_{2,h}(\mathcal{C})$.  Then there exists $T'\in\Psi^{N,-s}_{2,h}(\mathcal{C})$ such that $TT' + Q = \mathrm{Id}, \ Q\in\Re^0$.  Thus,
\[
\left\langle Pf,g \right\rangle_{L^2(\mathbb{T}^n)} = \left\langle T'Pf,T^* g \right\rangle_{L^2(\mathbb{T}^n)} + \left\langle QPf,g \right\rangle_{L^2(\mathbb{T}^n)}.
\]
We have
\[
\left\langle QPf,g \right\rangle_{L^2(\mathbb{T}^n)} = \left\langle Pf,Q^* g \right\rangle_{L^2(\mathbb{T}^n)} = \left\langle h^s Pf,h^{-s} Q^* g \right\rangle_{L^2(\mathbb{T}^n)}.
\]
Since $h^s P\in\Re^r$, $h^{-s} Q^*\in\Re^s$, then $h^s Pf, h^{-s} Q^* g\in I^\infty_{(0)}(\mathcal{C})\subset L^2(\mathbb{T}^n)$.  Thus,
\[
|\left\langle h^s Pf,h^{-s} Q^* g \right\rangle_{L^2(\mathbb{T}^n)}| \leq \|h^s Pf\|_{L^2(\mathbb{T}^n)} \|h^{-s} Q^* g\|_{L^2(\mathbb{T}^n)} \lesssim \|f\|_{I^\alpha_{(r)}} \|g\|_{I^\beta_{(s)}}.
\]
At the same time, $T'Pf\in I^\infty_{(0)}(\mathcal{C})\subset L^2(\mathbb{T}^n)$, and $T^* g\in L^2(\mathbb{T}^n)$ for all $N \geq -\beta$, so
\[
|\left\langle T'Pf,T^* g \right\rangle_{L^2(\mathbb{T}^n)}| \leq \|T'Pf\|_{L^2(\mathbb{T}^n)} \|T^* g\|_{L^2(\mathbb{T}^n)} \lesssim \|f\|_{I^\alpha_{(r)}} \|g\|_{I^\beta_{(s)}}. \qedhere
\]
\end{proof}

Fix $s\in\mathbb{R}$.  Let $P\in\Psi^{m,l}_{2,h}(\mathcal{C})$, and suppose that $p_0 = {}^2 \sigma_{m,l}(P)$ is real valued.  Let $H_{p_0}$ be the corresponding Hamiltonian vector field.  We first show that if $u\in I^{-\infty}_{(s)}(\mathcal{C})$, and if $u$ satisfies the equation $Pu = f$,
then ${}^2\mathrm{WF}(u) \backslash {}^2\mathrm{WF}(f)$ propagates along null bicharacteristics.

\begin{remark} \label{rmk:scaling}
$S_\mathrm{pr}$ is a manifold with boundary $SN(\mathcal{C})$, so $H_{p_0}$ must be rescaled to be tangent to this boundary.  If $\rho_\mathrm{ff}$ is a defining function for $SN(\mathcal{C})$, the rescaled vector field $\rho_\mathrm{ff}^{l-m+1} H_{p_0}$ is tangent to $SN(\mathcal{C})$.  In particular, if any point in an orbit of $\rho_\mathrm{ff}^{l-m+1} H_{p_0}$ lies in $SN(\mathcal{C})$, then the entire orbit lies in $SN(\mathcal{C})$.
\end{remark}

\begin{theorem} \label{primary propagation}
For $P\in\Psi^{m,l}_{2,h}(\mathcal{C})$, if $Pu = f$ and $p_0$ is real valued, then for any $k\in\mathbb{R}$
\[
{}^2\mathrm{WF}^{k,s}(u) \backslash {}^2\mathrm{WF}^{k-m+1,s-l}(f)
\]
propagates along the flow of $\rho_\mathrm{ff}^{l-m+1} H_{p_0}$.
\end{theorem}

\begin{proof}
First, note that
\[
{}^2 \mathrm{WF}^{k,s}(u) \backslash {}^2 \mathrm{WF}^{k-m,s-l}(f) \subset {}^2 \mathrm{char}(P),
\]
due to Corollary \ref{wavefront is characteristic}.

Let $\overline{H}_{p_0} = \rho_\mathrm{ff}^{l-m+1} H_{p_0}$.  Away from the boundary $SN (\mathcal{C})$, this result is no different from the usual semiclassical real principal type propagation.  For this reason, we take a point $q\in SN (\mathcal{C})$ (at which $\overline{H}_{p_0}$ does not vanish, i.e., is not radial).

Since $u\in I^{-\infty}_{(s)}(\mathcal{C})$, there exists $\beta$ for which ${}^2 \mathrm{WF}^{\beta,s}(u) = \emptyset$.
(More precisely, if $u\in I^\gamma_{(s)}(\mathcal{C})$, then ${}^2 \mathrm{WF}^{\beta,s}(u) = \emptyset$ for any $\beta\leq\gamma$.)  For some $\alpha > \beta$, assume that ${}^2 \mathrm{WF}^{\alpha - 1/2,s}(u)$ propagates along $\overline{H}_{p_0}$ flow.  Moreover, assume absence of ${}^2 \mathrm{WF}^{\alpha,s}(u)$ at one end of a bicharacteristic segment (the end opposite $q$):
\[
\exp{(t_0 \overline{H}_{p_0})} q \notin {}^2 \mathrm{WF}^{\alpha,s}(u)
\]
for some small $t_0 > 0$.  This implies, by \eqref{eqn:containment}, that $\exp{(t_0 \overline{H}_{p_0})} q \notin {}^2 \mathrm{WF}^{\alpha - 1/2,s}(u)$.  Therefore,
\[
\exp{(t \overline{H}_{p_0})} q \notin {}^2 \mathrm{WF}^{\alpha - 1/2,s}(u), \ t\in [0,t_0].
\]
Finally, assume absence of ${}^2 \mathrm{WF}^{\alpha-m+1,s-l}(f)$ along the whole segment:
\begin{equation} \label{assumption on inhomogeneity}
\exp{(t \overline{H}_{p_0})} q \notin {}^2 \mathrm{WF}^{\alpha-m+1,s-l}(f), \ t\in [0,t_0].
\end{equation}
We will then show that
\[
\exp{(t \overline{H}_{p_0})} q \notin {}^2 \mathrm{WF}^{\alpha,s}(u), \ t\in [0,t_0].
\]
The idea is that $\alpha$ is increased, in increments of a half, until it reaches $k$.  (It may be necessary to make minor numerological adjustments so that $\alpha$ equals $\beta$ plus an integral multiple of one-half, and $k$ is $\alpha$ plus an integral multiple of one-half.)

Next, define $\chi_0(s) = 0$ for $s \leq 0$, $\chi_0(s) = e^{-M/s}$ for $s > 0$ ($M>0$ to be specified); so $\chi_0$ is increasing on $(0,\infty)$ (but bounded above by one).  Let $\chi \geq 0$ be a smooth function with $\chi(s) = 0$ for $s\leq 0$ which increases to $\chi(s) = 1$ for $s \geq 1$, with $\sqrt{\chi}$ and $\sqrt{\chi'}$ both smooth.  Finally, let $\varphi$ be any smooth cutoff supported in $(-1,1)$.

We can choose coordinates $\rho_1,\ldots,\rho_{2n}$ on $S_\mathrm{pr}$ centered at $q$ in which (i) $\overline{H}_{p_0} = \partial_{\rho_1}$ and (ii) $SN (\mathcal{C}) = \{\rho_{2n}=0\}$, so that we may take $\rho_{2n}$ as front face boundary defining function.  Set $\rho'=(\rho_2,\ldots,\rho_{2n})$.  So, we assumed $\alpha,s$-regularity of $u$ at the point $(\rho_1=t_0,\rho'=0)$, and the other end of the piece of bicharacteristic is $q=(\rho_1=0,\rho'=0)$.  Let $\lambda$ be a positive parameter, and define the principal symbol of the commutant as follows:
\begin{align*}
a_\lambda(\rho_1,\rho')&=\rho_{2n}^{-s+l/2+\alpha-(m-1)/2} \varphi^2(\lambda^2 |\rho'|^2) \ \chi_0(\lambda\rho_1 + 1) \ \chi(\lambda(t_0 - \rho_1)-1)\in S^{s-l/2-\alpha+(m-1)/2}(S_\mathrm{pr}).
\end{align*}
Then on the support of $a$, we have $|\rho'| < \lambda^{-1}$ and $\rho_1 \geq -\lambda^{-1}$, $\rho_1 \leq t_0 - \lambda^{-1}$.  If the parameter $\lambda > 0$ is taken to be large, then $a$ is supported in a neighborhood of the bicharacteristic segment.  By the short exact sequence for second principal symbol, there exists $A\in\Psi^{\alpha-(m-1)/2, s - l/2}_{2,h}(\mathcal{C})$ which has (real valued) principal symbol $a$ and such that ${}^2 \mathrm{WF}_{s - l/2}'(A) = \mathrm{supp}(a)$.

``Commuting'' this $A$ with $P$,
\begin{align*}
{}^2 \sigma_{2\alpha,2s}(-i(A^* AP - P^* A^* A)) &= {}^2 \sigma_{2\alpha,2s}(-i[A^* A,P]) + {}^2 \sigma_{2\alpha,2s}(-i(P - P^*)A^* A) \\
&= \overline{H}_{p_0}\left(a^2\right) - i \ {}^2 \sigma_{m-1,l}(P - P^*) \ {}^2 \sigma_{2\alpha - m + 1,2s - l}(A^* A) \\
&= b^2 - e^2 + g a^2.
\end{align*}


(The $2\alpha+1,2s$-principal symbol of the commutator vanishes.)  Also, notice that the $m,l$-principal symbol of $P-P^*$ vanishes, since $p_0$ is real valued.  Above, $b^2,e^2\in S^{2s-2\alpha}(S_\mathrm{pr})$ arise when $\overline{H}_{p_0} = \partial_{\rho_1}$ is applied to $\chi_0^2(\lambda\rho_1 + 1)$ and $\chi^2(\lambda(t_0 - \rho_1)-1)$, respectively:
\begin{align*}
b^2 &= 2\lambda \ \rho_\mathrm{ff}^{-2s+l+2\alpha-m+1} \varphi^4\left(\lambda^2 |\rho'|^2\right) \ \chi_0(\lambda\rho_1 + 1) \ \chi_0'(\lambda\rho_1 + 1) \ \chi^2(\lambda(t_0 - \rho_1) - 1) \\
e^2 &= -2\lambda \ \rho_\mathrm{ff}^{-2s+l+2\alpha-m+1} \varphi^4\left(\lambda^2 |\rho'|^2\right) \ \chi_0^2(\lambda\rho_1 + 1) \ \chi(\lambda(t_0 - \rho_1)-1) \ \chi'(\lambda(t_0 - \rho_1)-1)
\end{align*}
Note that since $\chi(\lambda(t_0 - \rho_1)-1)$ is decreasing for $\rho_1 \in [t_0 - 2\lambda^{-1},t_0 - \lambda^{-1}]$ (and otherwise constant), $\chi'(\lambda(t_0 - \rho_1)-1)$ is negative in that interval, which means $e^2$ is actually positive.  So for large $\lambda$, $e$ is supported in the complement of ${}^2 \mathrm{WF}^{\alpha,s}(u)$.  Our assumption that $\sqrt{\chi}$ and $\sqrt{\chi'}$ are smooth ensures that $e\in S^{s-\alpha}(S_\mathrm{pr})$.  Conversely, since $\chi_0(\lambda\rho_1 + 1)$ is increasing for $\rho_1 \in (-\lambda^{-1},\infty)$, $\chi_0'(\lambda\rho_1 + 1)$ is positive in that interval (but $b^2$ is ``turned off'' at $t_0 - \lambda^{-1}$ by the $\chi^2$ term); hence, $b$ is supported along the whole segment.

Thus, again by the short exact sequence,
\begin{equation} \label{remainder}
-i(A^* AP - P^* A^* A) = B^* B + E^* E + A^* GA + R,
\end{equation}
where $R\in\Psi^{2\alpha - 1,2s}_{2,h}(\mathcal{C})$ satisfies ${}^2 \mathrm{WF}_{2s}'(R)\subset {}^2 \mathrm{WF}_{s - l/2}'(A)$, $B\in\Psi^{\alpha,s}_{2,h}(\mathcal{C})$ satisfies ${}^2\sigma_{\alpha,s}(B) = b$ and ${}^2 \mathrm{WF}_s'(B) = \mathrm{supp}(b)$, likewise for $E\in\Psi^{\alpha,s}_{2,h}(\mathcal{C})$; and $G\in\Psi^{m - 1,l}_{2,h}(\mathcal{C})$ has second principal symbol $g$ ($g$ may not be real valued).  By construction, ${}^2 \mathrm{WF}_s'(E)$ is contained in $|\rho'| < \lambda^{-1}$, $\rho_1\in[t_0 - 2\lambda^{-1},t_0 - \lambda^{-1}]$, so for large $\lambda$, ${}^2 \mathrm{WF}_s'(E)$ is contained in the complement of ${}^2 \mathrm{WF}^{\alpha,s}(u)$; thus $\|Eu\|_{L^2(\mathbb{T}^n)} < \infty$.  Similarly, ${}^2 \mathrm{WF}_{2s}'(R)$ is contained in $|\rho'| < \lambda^{-1}$, $\rho_1\in[-\lambda^{-1},t_0 - \lambda^{-1}]$, so for large $\lambda$, ${}^2 \mathrm{ell}(R) \subset {}^2 \mathrm{WF}_{2s}'(R)$ is disjoint from ${}^2 \mathrm{WF}^{\alpha - 1/2,s}(u)$; hence $|\left\langle Ru,u \right\rangle_{L^2(\mathbb{T}^n)}|$ is bounded as well.  On the other hand, ${}^2 \mathrm{WF}_s'(B)$ is contained in $|\rho'| < \lambda^{-1}$, $\rho_1\in[-\lambda^{-1},t_0 - \lambda^{-1}]$,
so \emph{a priori} $\|Bu\|_{L^2(\mathbb{T}^n)}$ is unbounded.

Pairing both sides of equation \eqref{remainder} with the distribution $u$, and using $Pu = f$, we have
\begin{equation} \label{ineq:inequality}
\|Bu\|^2_{L^2(\mathbb{T}^n)} \leq |\left\langle Ru,u \right\rangle_{L^2(\mathbb{T}^n)}| \ + \ |\left\langle Au,G^* Au \right\rangle_{L^2(\mathbb{T}^n)}| \ + 2 \ |\left\langle Au,Af \right\rangle_{L^2(\mathbb{T}^n)}| \ + \|Eu\|^2_{L^2(\mathbb{T}^n)}.
\end{equation}
Here we have used
\[
\left\langle (P^*A^*A - A^*AP)u,u \right\rangle_{L^2(\mathbb{T}^n)} = 2i \ \mathrm{Im}\left\langle Au,APu \right\rangle_{L^2(\mathbb{T}^n)}.
\]
We have already showed that the first and last terms on the right side of \eqref{ineq:inequality} are uniformly bounded as $h\downarrow 0$.

Let $T\in\Psi^{(m-1)/2,l/2}_{2,h}(\mathcal{C})$ be globally elliptic, with parametrix $T'$: $T' T + Q = \mathrm{Id}$ for some $Q\in\Re^0$.  We have
\begin{equation} \label{the other term}
|\left\langle Au,G^* Au \right\rangle_{L^2(\mathbb{T}^n)}| \leq \|TAu\|^2_{L^2(\mathbb{T}^n)} + \|(T')^* G^* Au\|^2_{L^2(\mathbb{T}^n)} + |\left\langle QAu,G^* Au \right\rangle_{L^2(\mathbb{T}^n)}|.
\end{equation}
We may control $\left\langle QAu,G^* Au \right\rangle_{L^2(\mathbb{T}^n)}$ by application of Lemma \ref{lem:bounding}.  Also, the principal symbols of $TA,(T')^* G^* A\in\Psi^{\alpha,s}_{2,h}(\mathcal{C})$ are multiples of $a$, so:

\textit{Claim:} The first two terms on the RHS of \eqref{the other term} may be absorbed into $\|Bu\|^2_{L^2(\mathbb{T}^n)}$, for sufficiently large $M$.  Comparing $a^2$ and $b^2$, it will suffice to show that
\[
\chi_0 (\lambda \rho_1 + 1) \leq 2 \lambda \chi_0'(\lambda \rho_1 + 1) \Longleftrightarrow (\log\chi_0)' \geq \frac{1}{2\lambda}.
\]
This amounts to finding $M$ for which $M\geq (\lambda \rho_1 + 1)^2 / (2\lambda)$.  Since $0\leq \rho_1 \leq t_0$, for all $\lambda$, no matter how large, there exists such an $M$.

Also, Cauchy--Schwarz gives
\begin{equation} \label{inhomogeneity}
|\left\langle Au,Af \right\rangle_{L^2(\mathbb{T}^n)}| \leq \|TAu\|^2_{L^2(\mathbb{T}^n)} + \|(T')^* Af\|^2_{L^2(\mathbb{T}^n)} + |\left\langle QAu,Af \right\rangle_{L^2(\mathbb{T}^n)}|.
\end{equation}


$\|TAu\|^2_{L^2(\mathbb{T}^n)}$ can be absorbed into $\|Bu\|^2_{L^2(\mathbb{T}^n)}$, as justified by the above claim.
$\|(T')^* Af\|^2_{L^2(\mathbb{T}^n)}$ is controlled as a result of assumption \eqref{assumption on inhomogeneity}, as $(T')^* A\in\Psi^{\alpha-m+1,s-l}_{2,h}(\mathcal{C})$.  We may directly apply Lemma \ref{lem:bounding} to control $\left\langle QAu,Af \right\rangle_{L^2(\mathbb{T}^n)}$.


Thus, $\|Bu\|_{L^2(\mathbb{T}^n)}<\infty$, so ${}^2 \mathrm{WF}^{\alpha,s}(u)$ is absent on the whole bicharacteristic segment.
\end{proof}

\begin{example}
We explicitly write down the vector field $H_{p_0}$ in a simple case.  The example is $\mathcal{C} = \{\xi_1=\xi_2=0\}$ in $T^* \mathbb{T}^3$.  The coordinates on $S_\mathrm{pr}$ are $x,\theta,\xi_3$, and $\rho_\mathrm{ff} = \sqrt{\xi_1^2 + \xi_2^2}$, where
\[
\cos(\theta) = \frac{\xi_1}{\sqrt{\xi_1^2 + \xi_2^2}}, \ \ \sin(\theta) = \frac{\xi_2}{\sqrt{\xi_1^2 + \xi_2^2}}.
\]
These coordinates\footnote{In codimension 3 and higher, we would not be able to use cosine and sine, but instead projective coordinates} are valid up to the boundary of $S_\mathrm{pr}$, namely $SN(\mathcal{C}) = \{\rho_\mathrm{ff} = 0\}$.

Let $\omega = \sum_{j=1}^3 dx_j\wedge d\xi_j$ denote the standard symplectic form on $T^* \mathbb{T}^3$.  The blowdown map is
\[
\beta: S_\mathrm{pr}\longrightarrow T^* \mathbb{T}^3, \ \ \beta(x,\rho_\mathrm{ff},\theta,\xi_3) = (x,\rho_\mathrm{ff}\cos\theta,\rho_\mathrm{ff}\sin\theta,\xi_3).
\]

Then the Hamiltonian vector field of $p_0\in C^\infty(S_\mathrm{pr}) = S^0 (S_\mathrm{pr})$ with respect to $\beta^* \omega$ is
\begin{equation} \label{Hamiltonian vf}
H_{p_0} = \left(\frac{\partial p_0}{\partial\rho_\mathrm{ff}} \cos\theta - \frac{1}{\rho_\mathrm{ff}} \frac{\partial p_0}{\partial\theta} \sin\theta\right) \frac{\partial}{\partial x_1}
+ \left(\frac{\partial p_0}{\partial\rho_\mathrm{ff}} \sin\theta + \frac{1}{\rho_\mathrm{ff}} \frac{\partial p_0}{\partial\theta} \cos\theta\right) \frac{\partial}{\partial x_2}
+ \frac{\partial p_0}{\partial\xi_3} \frac{\partial}{\partial x_3} +
\end{equation}
\[
+ \frac{1}{\rho_\mathrm{ff}} \left(\frac{\partial p_0}{\partial x_1} \sin\theta - \frac{\partial p_0}{\partial x_2} \cos\theta\right) \frac{\partial}{\partial\theta}
- \left(\cos\theta \frac{\partial p_0}{\partial x_1} + \sin\theta \frac{\partial p_0}{\partial x_2}\right) \frac{\partial}{\partial\rho_\mathrm{ff}} - \frac{\partial p_0}{\partial x_3} \frac{\partial}{\partial\xi_3}.
\]
Note that $\rho_\mathrm{ff} H_{p_0}$ is well defined up to the boundary.  \emph{However, as explained below, it may not be necessary to multiply by $\rho_\mathrm{ff}$ to ensure tangency to the boundary.}

Moreover, if $p_0(x,\xi_1,\xi_2,\xi_3)$ is smooth (in particular, smooth in $\xi_1$ and $\xi_2$), then since $\xi_1 = \rho_\mathrm{ff}\cos\theta$, $\xi_2 = \rho_\mathrm{ff}\sin\theta$, we have
\begin{equation} \label{cancellation}
\frac{\partial p_0}{\partial\theta} = \rho_\mathrm{ff}\left(\cos\theta\frac{\partial p_0}{\partial\xi_2} - \sin\theta\frac{\partial p_0}{\partial\xi_1}\right).
\end{equation}
Notice that $\rho_\mathrm{ff}$ cancels with $\rho_\mathrm{ff}^{-1}$ in the coefficients of $\partial_{x_1}$ and $\partial_{x_2}$ above.


\end{example}




\subsection{Principal type propagation, version 2}

We start with smooth, real valued principal symbol $p_0(\xi)$ depending only on the fiber variables in $T^* \mathbb{T}^n$.
Then the corresponding Hamiltonian vector field $H_{p_0}$ is \emph{a priori} tangent to the boundary $SN(\mathcal{C})$.  This can be seen in \eqref{Hamiltonian vf} by setting $\partial_{x_j} p_0 = 0$, together with \eqref{cancellation}.  Thus, even without rescaling, $H_{p_0}$ is well defined up to $SN(\mathcal{C})$.

So we have the following theorem, with the same proof as that of \autoref{primary propagation}:

\begin{theorem} \label{take two}
More generally, let $u\in I^{-\infty}_{(s)}(\mathcal{C})$.  For $P\in\Psi^{m,l}_{2,h}(\mathcal{C})$ with real principal symbol $p_0$, suppose $Pu = O_{L^2(\mathbb{T}^n)}(h^\infty)$.  If $p_0$ depends only on $\xi$, and is smooth in $\xi$, then for any $k$ ${}^2 \mathrm{WF}^{k,s}(u)$ propagates along the flow lines of $H_{p_0}$.
\end{theorem}

The key step in the proof of \autoref{primary propagation} is the ability to choose coordinates $\rho_j$ satisfying (i) $\overline{H}_{p_0} = \partial_{\rho_1}$ and (ii) $SN(\mathcal{C}) = \{\rho_{2n}=0\}$.  This is possible precisely because $\overline{H}_{p_0} = \rho_\mathrm{ff}^{l-m+1} H_\mathrm{p_0}$ is tangent to $SN (\mathcal{C})$.  Thus, if we assume $p_0 = p_0(\xi)$ (and choose not to rescale), then $H_\mathrm{p_0}$ is again tangent to $SN(\mathcal{C})$, and this choice of coordinates is again possible.

\begin{remark}
(1) For \emph{any} $p_0\in C^\infty(S_\mathrm{pr})$ (whether or not it is smooth `downstairs', whether or not it is independent of $x$), we may decompose $H_{p_0}$ as follows:
\begin{equation} \label{vf decomposition}
H_{p_0} = \vec{V}_1 + \frac{1}{\rho_\mathrm{ff}} \vec{V}_2,
\end{equation}
where $\vec{V}_1$ and $\vec{V}_2$ are each smooth up to the boundary $SN(\mathcal{C}) = \{\rho_\mathrm{ff} = 0\}$.  This decomposition is not canonical, however.  To see this, notice that we may further decompose $\vec{V}_2$ as $\vec{V}_2 = \vec{Y}_1 + \rho_\mathrm{ff} \vec{Y}_2$.  Thus,
\begin{equation} \label{vf decomposition 2}
H_{p_0} = (\vec{V}_1 + \vec{Y}_2) + \rho_\mathrm{ff}^{-1} \vec{Y}_1,
\end{equation}
where $\vec{V}_1 + \vec{Y}_2$, $\vec{Y}_1$ are each smooth up to $SN(\mathcal{C})$.  While the role of $\vec{V}_2$ is therefore not unique, we always have
\[
{\vec{V}_2}|_{SN(\mathcal{C})} = {\vec{Y}_1}|_{SN(\mathcal{C})} = (\rho_\mathrm{ff} H_{p_0})|_{SN(\mathcal{C})}.
\]

(2)
If we rescale: in \autoref{primary propagation}, second wavefront is propagating along $\rho_\mathrm{ff} H_{p_0} = \rho_\mathrm{ff} \vec{V}_1 + \vec{V}_2$.
(If we just wanted to prove propagation away from the boundary of $S_\mathrm{pr}$, there was no need to rescale at all.  Away from $SN(\mathcal{C})$, rescaling merely reparametrizes the same flow lines.)

(3) If instead we assume $p_0 = p_0(\xi)$ and smoothness in $\xi$, then as we have showed, $H_{p_0}$ need not be rescaled.
Under these assumptions, in the example,
\[
\vec{Y}_1 = \left(\frac{\partial p_0}{\partial x_1} \sin\theta - \frac{\partial p_0}{\partial x_2} \cos\theta\right) \frac{\partial}{\partial\theta}, \ \ \vec{Y}_2 = \left(\cos\theta\frac{\partial p_0}{\partial\xi_2} - \sin\theta\frac{\partial p_0}{\partial\xi_1}\right)\left(\cos\theta\frac{\partial}{\partial x_2} - \sin\theta\frac{\partial}{\partial x_1}\right).
\]
And as we see, this $\vec{Y}_1$ is zero when $p_0 = p_0(\xi)$, so the problematic term $\vec{Y}_1 / {\rho_\mathrm{ff}}$ in \eqref{vf decomposition 2} drops out.  So then $H_{p_0} = \vec{V}_1 + \vec{Y}_2$, and \autoref{take two} gives propagation along this vector field.



\end{remark}

\subsection{Secondary propagation of coisotropic wavefront\label{secondary propagation subsection}}

Let $P\in\widetilde{\Psi}^{0}_{h}(\mathbb{T}^n)$ have real valued principal symbol $p_0$.  Recall from Remark \ref{one calculus in the other} that $P$ can be regarded as a second microlocal operator: $P\in\widetilde{\Psi}^{0}_{h}(\mathbb{T}^n)\subset\Psi^{0,0}_{2,h}(\mathcal{C})$.
We will abuse notation and refer to the second principal symbol of $P$ also as $p_0$.

The linear coisotropic $\mathcal{C}$ is given by
\[
\mathbb{T}^n_x \times \{\mathbf{v}_1 \cdot \xi = \ldots = \mathbf{v}_d \cdot \xi = 0\}.
\]
for linearly independent $\{\mathbf{v}_1,\ldots,\mathbf{v}_d\}\subset\mathbb{R}^n$.  Let $\widetilde{\xi}' = (\mathbf{v}_1 \cdot \xi,\ldots,\mathbf{v}_d \cdot \xi)$.  We extend $\{\mathbf{v}_1,\ldots,\mathbf{v}_d\}$ to a basis
\[
\{\mathbf{v}_1,\ldots,\mathbf{v}_d,\mathbf{w}_{d+1},\ldots,\mathbf{w}_n\}
\]
for $\mathbb{R}^n$, as before.  Let $\widetilde{\xi}'' = (\mathbf{w}_{d+1} \cdot \xi,\ldots,\mathbf{w}_n \cdot \xi)$ and $\widetilde{\xi} = (\widetilde{\xi}',\widetilde{\xi}'')$.  Locally, we may define $\widetilde{x}'$, $\widetilde{x}''$, and $\widetilde{x}$ analogously.  In the coordinates $(\widetilde{x},\widetilde{\xi})$, $\mathcal{C} = \mathbb{T}^n_{\widetilde{x}} \times \{\widetilde{\xi}' = 0\}$.

Further assume that $p_0$ is a function only of $\widetilde{\xi}$, i.e., is independent of $\widetilde{x}$.  Hence, $H_{p_0}$ is tangent to $\mathcal{C}$.  Set $\rho_\mathrm{ff} := |\widetilde{\xi}'|$, and $\Gamma' := \widetilde{\xi}' / |\widetilde{\xi}'|$. As the notation suggests, $\rho_\mathrm{ff}$ is a front face defining function: $\partial S_\mathrm{pr} = SN(\mathcal{C}) = \{\rho_\mathrm{ff} = 0\}$.

For our next result, we impose a further condition, on the subprincipal symbol of $P$.  Locally, we may write $P = {}^h \mathrm{Op_W}(p)$, where the total symbol $p$ decomposes as
\[
p = p_0 + h\mathrm{sub}(P) + O(h^2).
\]
(We specified the Weyl quantization here, so that the subprincipal symbol is the $O(h)$ term of the total symbol.) We assume that $\mathrm{sub}(P) \equiv 0$.

We remarked earlier that $P$ can be regarded as an element of $\Psi^{0,0}_{2,h}(\mathcal{C})$.  We can Taylor expand $p_0$ at $\mathcal{C}$, partially in the characteristic variables $\widetilde{\xi'}$, to obtain
\[
{p_0}\restriction_\mathcal{C}\left(\widetilde{\xi}''\right) + \left.\frac{\partial p_0}{\partial\widetilde{\xi}'}\right|_\mathcal{C}\left(\widetilde{\xi}''\right)\cdot\widetilde{\xi}' + \frac{1}{2}\left.\frac{\partial^2 p_0}{{\partial\widetilde{\xi}'}^2}\right|_\mathcal{C}\left(\widetilde{\xi}''\right) \cdot \left(\widetilde{\xi}'\right)^2 + O\left(\left(\widetilde{\xi}'\right)^3\right).
\]
Note that $p_0\restriction_\mathcal{C}\left(\widetilde{\xi}''\right) = p_0\left(0,\widetilde{\xi}''\right)$.
The corresponding Hamiltonian vector field, computed term-by-term, is
\small
\[
H_{p_0} = \frac{\partial {p_0}\restriction_\mathcal{C}}{\partial\widetilde{\xi}''}\cdot\partial_{\widetilde{x}''} + \left[\left(\left.\frac{\partial^2 p_0}{\partial\widetilde{\xi}'\partial\widetilde{\xi}''}\right|_\mathcal{C}\left(\widetilde{\xi}''\right)\right) \partial_{\widetilde{x}''}\right]\cdot\widetilde{\xi}' + \left.\frac{\partial p_0}{\partial\widetilde{\xi}'}\right|_\mathcal{C}\left(\widetilde{\xi}''\right)\cdot\partial_{\widetilde{x}'} + \left[\left(\left.\frac{\partial^2 p_0}{{\partial\widetilde{\xi}'}^2}\right|_\mathcal{C}\left(\widetilde{\xi}''\right)\right) \widetilde{\xi}'\right]\cdot\partial_{\widetilde{x}'} + O\left(\left(\widetilde{\xi}'\right)^2\right).
\]
\normalsize
In the coordinates of $S_\mathrm{pr}$, $H_{p_0}$ is lifted to
\begin{align}
\mathbf{H} &= \left.\frac{\partial p_0}{\partial\widetilde{\xi}''}\right|_{SN(\mathcal{C})}\left(\widetilde{\xi}''\right)\cdot\partial_{\widetilde{x}''} + \left.\frac{\partial p_0}{\partial\widetilde{\xi}'}\right|_{SN(\mathcal{C})}\left(\widetilde{\xi}''\right)\cdot\partial_{\widetilde{x}'} + \label{Hamilton} \\
&+ \rho_\mathrm{ff}\left(\left[\left(\left.\frac{\partial^2 p_0}{\partial\widetilde{\xi}'\partial\widetilde{\xi}''}\right|_{SN(\mathcal{C})}\left(\widetilde{\xi}''\right)\right) \partial_{\widetilde{x}''}\right]\cdot\Gamma' + \left[\left(\left.\frac{\partial^2 p_0}{{\partial\widetilde{\xi}'}^2}\right|_{SN(\mathcal{C})}\left(\widetilde{\xi}''\right)\right) \Gamma'\right]\cdot\partial_{\widetilde{x}'}\right) + \rho_\mathrm{ff}^2 H'. \nonumber
\end{align}
$H'$ is tangent to $SN (\mathcal{C}) = \{\rho_\mathrm{ff} = 0\}$, as it does not contain $\partial_{\widetilde{\xi}}$.  Let $H_1$ refer to the sum of the first two terms of $\mathbf{H}$, and $H_2$ to the $O(\rho_\mathrm{ff})$-piece of $\mathbf{H}$ (i.e., $H_2$ is the next order jets to $H_1$).  Before stating the theorem, we give an example.


\begin{example} \label{ex:symbol class}
Suppose $a\in C^\infty_c(S_\mathrm{tot})$ for the coisotropic $\{\mathbf{v}_1 \cdot \xi = \mathbf{v}_2 \cdot \xi = 0\} \subset T^*\mathbb{T}^3$.  For later use, we explicitly determine the symbol class to which $h^{|\alpha|+|\beta|-1}\left(\partial^\alpha_{\widetilde{\xi}} \partial^\beta_{\widetilde{x}} a\right)$ belongs.  We employ the following coordinates in one coordinate patch of $S_\mathrm{tot}$: $\widetilde{x}$, $H = h / {\widetilde{\xi}_1}$, $\widetilde{\xi}_1$, $\Xi_2 = \widetilde{\xi}_2 / \widetilde{\xi}_1$, and $\widetilde{\xi}_3 = \mathbf{w}_3 \cdot \xi$.  $H$ is a defining function for the side face, $\widetilde{\xi}_1$ for the front face.  We compute $h\partial_{\widetilde{\xi}_1} = H(\widetilde{\xi}_1\partial_{\widetilde{\xi}_1} - H\partial_{\Xi_2} - H\partial_H)$.  The lifted vector field in parentheses is tangent to both side and front faces.  We also have $h\partial_{\widetilde{\xi}_2} = H\partial_{\Xi_2}$. Therefore,
\[
h^{|\alpha|+|\beta|-1}\left(\partial^\alpha_{\widetilde{\xi}} \partial^\beta_{\widetilde{x}} a\right) \in S^{1-|\alpha|-|\beta|,1-|\beta|-\alpha_3}(S_\mathrm{tot}).
\]
More generally, for $\{\mathbf{v}_1 \cdot \xi = \ldots = \mathbf{v}_d \cdot \xi = 0\}\subset T^*\mathbb{T}^n$,
\[
h^{|\alpha|+|\beta|-1}\left(\partial^\alpha_{\widetilde{\xi}} \partial^\beta_{\widetilde{x}} a\right)\in S^{1-|\alpha|-|\beta|,1-|\beta|-(\alpha_{d+1}+\ldots+\alpha_n)}(S_\mathrm{tot}).
\]
\end{example}

\begin{theorem} \label{secondary propagation}
Let $\mathcal{C}$ be a linear coisotropic submanifold.  Assume $P\in\widetilde{\Psi}^0_h(\mathbb{T}^n)$ has real valued principal symbol depending only on the fiber variables in $T^* \mathbb{T}^n$, and subprincipal symbol identically equal to zero.  Let $u \in L^2(\mathbb{T}^n)$ satisfy $Pu = O_{L^2(\mathbb{T}^n)}(h^\infty)$.  Then for all $l \leq 0$, ${}^2 \mathrm{WF}^{\infty,l}(u) \cap SN(\mathcal{C})$ propagates along the flow of $H_1$; and for all $l \leq -1$, ${}^2 \mathrm{WF}^{\infty,l}(u) \cap SN(\mathcal{C})$ propagates along the flow of both $H_1$ and $H_2$.
\end{theorem}

Since $u \in L^2(\mathbb{T}^n) = I^0_{(0)}(\mathcal{C})$ and $I^0_{(0)}(\mathcal{C})\subset I^0_{(l)}(\mathcal{C})$ for $l\leq 0$, it makes sense to consider the $m,l$-wavefront set of $u$ for any $m\in\mathbb{R}$.

\begin{proof}
The vector field $H_1$ is the same as $(H_{p_0})|_{SN(\mathcal{C})}$ of \autoref{take two} in the case $m = l = 0$.  Therefore, we obtain $H_1$ invariance by simply quoting \autoref{take two}.

Since the coordinates $(\widetilde{x},\widetilde{\xi})$ are valid locally in a neighborhood of $\mathcal{C}$, we may extend the vector field $H_1$ to a neighborhood of $SN(\mathcal{C})$.  For $\epsilon > 0$, we extend $H_1$ to $N := \{\rho_\mathrm{ff} = |\widetilde{\xi}'| < \epsilon\}$.  Let $\widetilde{H}_2$ be defined on $N$ by $\widetilde{H}_2 := \rho_\mathrm{ff}^{-1}(\mathbf{H} - H_1)$.  We see that $\widetilde{H}_2$ coincides with $H_2$ on $SN(\mathcal{C})$.  Since the principal symbol of $P$ is independent of $\widetilde{x}$, then $\widetilde{\xi}$ is constant along the flows of $H_1$ and $\widetilde{H}_2$, so the $(H_1,\widetilde{H}_2)$ joint flow from $N$ stays in this neighborhood of $SN(\mathcal{C})$ for all times.

Next, there exists $m_0\in\mathbb{R}$ for which ${}^2 \mathrm{WF}^{m_0,l}(u) = \emptyset$.  For some $m > m_0$, assume ${}^2 \mathrm{WF}^{m - 1/2,l}(u)$ is invariant under the $\widetilde{H}_2$ flow.  Take any $\zeta \in SN(\mathcal{C})$ and suppose that $\zeta \notin {}^2 \mathrm{WF}^{m,l}(u)$.  We seek to prove that (the closure of) the $\widetilde{H}_2$ orbit through $\zeta$ is disjoint from ${}^2 \mathrm{WF}^{m,l}(u)$.  This argument may then be iterated to show $\widetilde{H}_2$ invariance of ${}^2 \mathrm{WF}^{\infty,l}(u)$.  Let $\overline{\mathcal{O}}_{H_1}(\zeta)$ refer to (the closure of) the $H_1$ orbit through $\zeta$; due to $H_1$-invariance, we know that $\overline{\mathcal{O}}_{H_1}(\zeta) \cap {}^2 \mathrm{WF}^{m,l}(u)=\emptyset$.  Let $\overline{\mathcal{O}}_{\widetilde{H}_2}(\zeta)$ refer to (the closure of) the $\widetilde{H}_2$ orbit through $\zeta$; since $\zeta \notin {}^2 \mathrm{WF}^{m - 1/2,l}(u)$, we know $\overline{\mathcal{O}}_{\widetilde{H}_2}(\zeta) \cap {}^2 \mathrm{WF}^{m - 1/2,l}(u)=\emptyset$.

Since $\zeta\in\{\rho_\mathrm{ff} = 0\}$, $\rho_\mathrm{ff} = |\widetilde{\xi}'|$, and $\widetilde{\xi}$ is constant on $H_1$ orbits, the entire $H_1$ orbit containing $\zeta$ lies in $SN(\mathcal{C})$.  Since $\overline{\mathcal{O}}_{H_1}(\zeta) \cap {}^2 \mathrm{WF}^{m,l}(u)=\emptyset$, there exists a neighborhood of this orbit closure in $S_\mathrm{pr}$ that is disjoint from ${}^2 \mathrm{WF}^{m,l}(u)$.  Choose nonnegative $a_0\in C^\infty(S_\mathrm{pr})$ whose support contains $\overline{\mathcal{O}}_{H_1}(\zeta)$, whose support is contained in this neighborhood (so is disjoint from ${}^2 \mathrm{WF}^{m,l}(u)$), whose support is in $N$ (see previous paragraph), and such that $H_1(a_0) = 0$.  Note that in particular, $\mathrm{supp}(a_0)$ is disjoint from ${}^2 \mathrm{WF}^{m - 1/2,l}(u)$.

Fix $\delta \in (0,1)$.  We consider $\widetilde{H}_2$ flow segments with one endpoint $\exp\left(0 \widetilde{H}_2\right)$ on the $H_1$ orbit through $\zeta$ (so one of these segments has endpoint $\zeta$ itself) and other endpoint $\exp\left(-\delta \widetilde{H}_2\right)$.  Informally, we are taking the $H_1$ orbit passing through $\zeta$ and ``smearing'' it along the $\widetilde{H}_2$ flow; see Figure \ref{fig:smear}.

Next, for $p \in N$, put
\[
a_1(p) := -\int^\delta_0 (1 - s) a_0\left(\exp\left(-s \widetilde{H}_2\right) p \right) ds.
\]
Away from $N$, we are certainly off the support of $a_0$, so we may define $a_1$ to be zero on $S_\mathrm{pr} \backslash N$; then $a_1\in C^\infty(S_\mathrm{pr})$.  We are most interested in $p\in\overline{\mathcal{O}}_{H_1}(\zeta)\subset SN (\mathcal{C})$.  We have
\begin{equation} \label{eqn:symbol support}
\mathrm{supp}(a_1)\supseteq \{\exp(-s \widetilde{H}_2) p \ | \ 0\leq s\leq\delta, \ p\in\overline{\mathcal{O}}_{H_1}(\zeta)\}.
\end{equation}


Since $\mathrm{supp}(a_0) \cap {}^2 \mathrm{WF}^{m - 1/2,l}(u) = \emptyset$ and ${}^2 \mathrm{WF}^{m - 1/2,l}(u)$ is $\widetilde{H}_2$ invariant, we have $\mathrm{supp}(a_1) \cap {}^2 \mathrm{WF}^{m - 1/2,l}(u) = \emptyset$.

\begin{figure}[htb]
 \centering
 \def\svgwidth{300pt}
 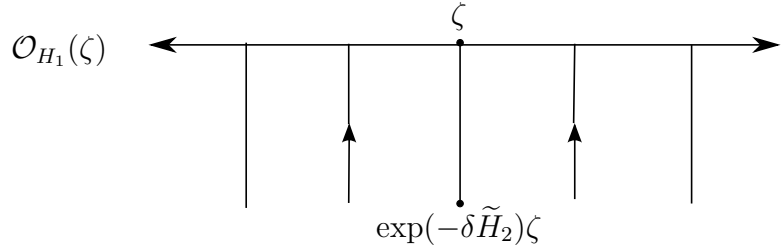
 \caption[Time $\delta$ backwards $\widetilde{H}_2$-flowout]{Time $\delta$ backwards $\widetilde{H}_2$ flowout to $H_1$ orbit through $\zeta$}
 \label{fig:smear}
\end{figure}

Then define $\widetilde{a}_1 := \rho_\mathrm{ff}^{-(2l+1)+2m} a_1 \in S^{2l+1-2m}(S_\mathrm{pr})$.  We compute
\begin{align*}
\rho_\mathrm{ff} \widetilde{H}_2(\widetilde{a}_1) &= \rho_\mathrm{ff}^{-2l+2m} \widetilde{H}_2(a_1) = \rho_\mathrm{ff}^{-2l+2m} \int^\delta_0 (1-s) \frac{\partial}{\partial s}\left[a_0\left(\exp\left(-s \widetilde{H}_2\right)\right)\right] ds \\
&= \rho_\mathrm{ff}^{-2l+2m} \int^\delta_0 a_0\left(\exp\left(-s \widetilde{H}_2\right)\right) ds + (1-\delta) \rho_\mathrm{ff}^{-2l+2m} a_0\left(\exp\left(-\delta \widetilde{H}_2\right)\right) - \rho_\mathrm{ff}^{-2l+2m} a_0 \\
&=: \int^\delta_0 b^2_s \ ds + c^2 - d^2.
\end{align*}
We have $[H_1,\widetilde{H}_2]=0$.  This, combined with $H_1(a_0)=0$, implies $H_1(a_1)=0$, which in turn implies $H_1(\widetilde{a}_1)=0$.

Choose $A\in\Psi^{2m,2l+1}_{2,h}(\mathcal{C})$ with principal symbol $\widetilde{a}_1$, such that ${}^2 \mathrm{WF}_{2l+1}'(A) = \mathrm{supp}(\widetilde{a}_1)$.  Locally, we may write $A := {}^h \mathrm{Op_W}(a)$ for $a\in S^{2m,2l+1}(S_\mathrm{tot})$.  Assume further that $a$ is real valued, hence $A$ is formally self-adjoint (and actually self-adjoint for $m\leq 0$, $l\leq -1$).  We also have $P = {}^h \mathrm{Op_W}(p)$ locally.  Since in particular $P\in\widetilde{\Psi}^0_h(\mathbb{T}^n)$, $p$ is smooth on the blown down space $T^*\mathbb{T}^n \times [0,1)_h$; this is not true of $a$.

We ``commute'' $h^{-1} P\in\Psi^{1,1}_{2,h}(\mathcal{C})$ with $A$ to obtain
\small
\[
i((h^{-1}P)^* A - A(h^{-1}P)) \sim {}^h \mathrm{Op_W}\left(\sum_{\alpha,\beta}\frac{ih^{-1} (-1)^{|\alpha|}}{(2i)^{|\alpha + \beta|}\alpha!\beta!}\left((\partial^\alpha_{\widetilde{x}} \partial^\beta_{\widetilde{\xi}} \bar{p})((h \partial_{\widetilde{\xi}})^\alpha \partial^\beta_{\widetilde{x}} a)-(\partial^\alpha_{\widetilde{x}} (h \partial_{\widetilde{\xi}})^\beta a)(\partial^\alpha_{\widetilde{\xi}} \partial^\beta_{\widetilde{x}} p)\right)\right) + R,
\]
\normalsize
where $R\in\Re^{2l+2}$.  This is the Weyl formula for the total symbol of a composition (cf.\ \cite[Section~2.7]{Ma}).

``Slicing'' this sum ($|\alpha + \beta|=0$, $|\alpha + \beta|=1$, $|\alpha + \beta|=2$, etc.) and examining each ``slice'' separately, we determine that
\[
{}^2 \sigma_{2m,2l+2}(i((h^{-1}P)^* A - A(h^{-1}P))) = \mathbf{H}(\widetilde{a}_1) + 2 \widetilde{a}_1 \mathrm{Im}(\mathrm{sub}(P)).
\]
(Note that the $2m+1,2l+2$-principal symbol of the ``commutator'' vanishes.)  The first summand $\mathbf{H}(\widetilde{a}_1)$ arises from the ``slice'' $|\alpha + \beta|=1$ by replacing $p$ with its real valued principal part $p_0$
and $a$ with $\widetilde{a}_1$.  The latter summand comes from the ``slice'' $\alpha=\beta=0$, replacing $p$ with $h\mathrm{sub}(P)$ and $a$ with $\widetilde{a}_1$.  Now, since $H_1(\widetilde{a}_1) = 0$, and since we assumed $\mathrm{sub}(P)\equiv 0$, we conclude that
\[
{}^2 \sigma_\mathrm{2m,2l+2}(i((h^{-1}P)^* A - A(h^{-1}P))) = \rho_\mathrm{ff} \widetilde{H}_2(\widetilde{a}_1).
\]
Thus, the $2m,2l+2$-principal symbol of the ``commutator'' vanishes to first order at $SN(\mathcal{C})$.

Let $B_s$, $C$, $D\in\Psi^{m,l}_{2,h}(\mathcal{C})$ have principal symbols $b_s$, $c$, $d\in S^{l-m}(S_\mathrm{pr})$, respectively, from earlier; and such that ${}^2 \mathrm{WF}_{l}'(B_s) = \mathrm{supp}(b_s)$, and likewise for $C,D$.


Thus, since $\int^\delta_0 B_s^* B_s \ ds + C^* C - D^* D$ and $i((h^{-1}P)^* A - A(h^{-1}P))$ share the same principal symbol, we must have
\begin{equation} \label{remainder 2}
i\left((h^{-1}P)^* A - A(h^{-1}P)\right) = \int^\delta_0 B_s^* B_s \ ds + C^* C - D^* D + L,
\end{equation}
where $L\in\Psi^{2m-1,2l+2}_{2,h}(\mathcal{C})$ satisfies ${}^2 \mathrm{WF}_{2l+2}'(L) \subset {}^2 \mathrm{WF}_{2l+1}'(A)$.  Note that $D$ is microsupported along the $H_1$ orbit containing $\zeta$, at which we have absence of ${}^2 \mathrm{WF}^{m,l}(u)$; and that $C$ lives at the opposite ends of the $H_2$ flow segments, so $C^* C$ is the term we wish to control.


Unfortunately, the decay of $L$ at the front face is two orders worse than that of $C^* C$.  To overcome this issue, we present the following lemma, to be followed by the end of the proof of the theorem.

\begin{lemma}[Decomposition of $L$]
$L\in\Psi^{2m-1,2l+2}_{2,h}(\mathcal{C})$ in Equation \eqref{remainder 2} may be decomposed as $L = L_1 + L_2$, where $L_1\in\Psi^{2m-1,2l}_{2,h}(\mathcal{C})$ and $L_2\in\Re^{2l+2}$.
\end{lemma}

\begin{proof}
We make full use of the assumption that $\mathrm{sub}(P)\equiv 0$.  We expand the total symbol $p$ (of $P = {}^h \mathrm{Op_W}(p)$) in powers of $h$:
\[
p(\widetilde{x},\widetilde{\xi}) = p_0(\widetilde{\xi}) + O\left(h^2\right) = p_0 + p_1,
\]
where $p_1$ signifies everything but the principal part of $P$.

Next, we again make use of the Weyl formula, in conjunction with equation \eqref{remainder 2}.  If we replace $p$ by $p_0$ in the asymptotic sum, and replace $a$ by $\widetilde{a}_1$, the term $\alpha=\beta=0$ directly cancels, and the term $|\alpha+\beta| = 1$ is quantized to give ${}^h \mathrm{Op_W}(\{p_0,\widetilde{a}_1\})\in\Psi^{2m,2l}_{2,h}(\mathcal{C})$.


This differs from $\int^\delta_0 B_s^* B_s \ ds + C^* C - D^* D$ by a remainder $\tilde{L}\in\Psi^{2m-1,2l}_{2,h}(\mathcal{C})$, since ${}^h \mathrm{Op_W}(\{p_0,\widetilde{a}_1\})$ and $\int^\delta_0 B_s^* B_s \ ds + C^* C - D^* D$ both have symbol $\mathbf{H}(\widetilde{a}_1) = \rho_\mathrm{ff} \widetilde{H}_2(\widetilde{a}_1)$.


Therefore, $L\in\Psi^{2m-1,2l+2}_{2,h}(\mathcal{C})$ is generated by quantizing all terms involving $p_1$, plus all terms with $|\alpha+\beta| > 1$ with $p_0$ in place of $p$, plus $\tilde{L}$ plus the residual operator $L_2 := R\in\Re^{2l+2}$.

First, we carefully study the terms in the Weyl expansion arising from $p=p_0 + p_1$ from all ``slices'' $|\alpha+\beta| > 1$.  To do this, we introduce local coordinates in one coordinate patch of $S_\mathrm{tot}$: $\widetilde{x}$, $\widetilde{\xi}_1$, $H := h / {\widetilde{\xi}_1}$, $\Xi_2 := \widetilde{\xi}_2 / \widetilde{\xi}_1,\ldots,\Xi_d := \widetilde{\xi}_d / \widetilde{\xi}_1$, and $\widetilde{\xi}''$.  As usual, $H$ is a defining function for the side face and $\widetilde{\xi}_1$ for the front face.  At this point, we use our earlier observation that $p$ is smooth on $T^*\mathbb{T}^n \times [0,1)$ (in particular $p$ is smooth in $\widetilde{\xi}'$, unlike $a$), so for all multi-indices $\gamma$ we have $\partial^\gamma_{\widetilde{\xi}} p = O(1)$ and $\partial^\gamma_{\widetilde{x}} p = O(1)$.  ``Slices'' for which $|\alpha+\beta| > 1$ satisfy (1) $|\alpha|\geq 2$ or (2) $|\beta|\geq 1$, $|\alpha|\leq |\beta|$ (these cases are not exclusive).  Example \ref{ex:symbol class} gives
\[
h^{|\alpha+\beta| - 1}\left(\partial^\alpha_{\widetilde{x}} \partial^\beta_{\widetilde{\xi}} a\right)\left(\partial^\alpha_{\widetilde{\xi}} \partial^\beta_{\widetilde{x}} p\right)\in S^{2m + 1-|\alpha+\beta|,2l+2-|\alpha|-(\beta_{d+1}+\ldots+\beta_n)}(S_\mathrm{tot}).
\]
If $\gamma\neq 0$, we can improve on $\partial^\gamma_{\widetilde{x}} p = O(1)$.  We have:
\begin{align}
p(\widetilde{x},\widetilde{\xi}) &= p_0(\widetilde{\xi}) + O\left(h^2\right) \nonumber \\
&= {p_0}\restriction_\mathcal{C}(\widetilde{\xi}'') + \rho_\mathrm{ff} \left[\frac{\partial p_0}{\partial\widetilde{\xi}'}\right]_\mathcal{C}(\widetilde{\xi}'')\cdot\Gamma' + \frac{\rho_\mathrm{ff}^2}{2} \left[\frac{\partial^2 p_0}{{\partial\widetilde{\xi}'}^2}\right]_\mathcal{C}(\widetilde{\xi}'')\cdot (\Gamma')^2 + O\left(\rho_\mathrm{ff}^3\right) + O\left(\rho_\mathrm{sf}^2 \rho_\mathrm{ff}^2\right), \label{eqn:Taylor expansion}
\end{align}
where $SN^+(\mathcal{C}\times\{0\}) = \{\rho_\mathrm{ff} = 0\}$.


The $O\left(\rho_\mathrm{ff}^3\right)$ term is the remainder in the Taylor expansion of the principal symbol, so it is independent of $\widetilde{x}$.  Then we differentiate \eqref{eqn:Taylor expansion} to get $\partial^\gamma_{\widetilde{x}} p = O\left(\rho^2_\mathrm{sf} \rho^2_\mathrm{ff}\right) = O\left(H^2 \widetilde{\xi}_1^2\right)$.  This implies, in case (2), that
\[
h^{|\alpha+\beta| - 1}\left(\partial^\alpha_{\widetilde{x}} \partial^\beta_{\widetilde{\xi}} a\right)\left(\partial^\alpha_{\widetilde{\xi}} \partial^\beta_{\widetilde{x}} p\right)\in S^{2m-1-|\alpha+\beta|,2l-|\alpha|-(\beta_{d+1}+\ldots+\beta_n)}(S_\mathrm{tot}).
\]
We are thus able to conclude that all terms of the form $h^{|\alpha+\beta| - 1}\left(\partial^\alpha_{\widetilde{x}} \partial^\beta_{\widetilde{\xi}} a\right)\left(\partial^\alpha_{\widetilde{\xi}} \partial^\beta_{\widetilde{x}} p\right)$ for which $|\alpha+\beta| > 1$ are $O\left(\rho_\mathrm{sf}^{-2m+1} \rho_\mathrm{ff}^{-2l}\right)$.

Finally, we consider the terms arising when $p = p_1 = O\left(\rho_\mathrm{sf}^2 \rho_\mathrm{ff}^2\right)$ and $|\alpha+\beta|\leq 1$.  If $\alpha=\beta=0$, it is easily seen that $h^{-1} a p_1\in S^{2m-1,2l}(S_\mathrm{tot})$.  Next suppose $\alpha=0$, $|\beta|=1$.  We assume the worst: one of the first $d$ components of $\beta$ is equal to one.  Then, since $\partial^\beta_{\widetilde{x}} p_1 = O\left(\rho_\mathrm{sf}^2 \rho_\mathrm{ff}^2\right)$ and $h^{-1}(h \partial_{\widetilde{\xi}})^\beta a\in S^{2m,2l+2}(S_\mathrm{tot})$, we have
\[
h^{-1} \left(\partial^\beta_{\widetilde{x}} p_1\right) \left(h\partial_{\widetilde{\xi}}\right)^\beta a\in S^{2m-2,2l}(S_\mathrm{tot}).
\]
The remaining case is $|\alpha|=1$, $\beta=0$.  We again assume the worst: $\alpha_1 = 1$.  Since $\partial_{\widetilde{\xi}_1} p_1 = O\left(H^2 \widetilde{\xi}_1\right)$ and $\partial_{\widetilde{x}_1} a\in S^{2m,2l+1}(S_\mathrm{tot})$, we have
\[
(\partial_{\widetilde{x}_1} a)(\partial_{\widetilde{\xi}_1} p_1)\in S^{2m-2,2l}(S_\mathrm{tot}),
\]
which is even better than we need.

A similar calculation holds if we interchange $\alpha$ and $\beta$ and take complex conjugates.  As a result, we may Borel sum, then quantize, then add on $\tilde{L}$ to obtain the desired operator $L_1\in\Psi^{2m-1,2l}_{2,h}(\mathcal{C})$.  This completes the proof of the lemma.
\end{proof}

First, using $A = A^*$, we find that
\[
\left\langle i((h^{-1}P)^* A - A(h^{-1}P))u,u \right\rangle_{L^2(\mathbb{T}^n)} = -2 \ \mathrm{Im}\left\langle Au,h^{-1}Pu \right\rangle_{L^2(\mathbb{T}^n)}.
\]
Then, applying the lemma to Equation \eqref{remainder 2},
\begin{equation} \label{eq:rearranged eqn}
\int^\delta_0 \|B_s u\|^2 ds + \|Cu\|^2 \leq \|Du\|^2 + \ |\left\langle L_1 u,u \right\rangle| \ + \ |\left\langle L_2 u,u \right\rangle| \ + 2 \ |\left\langle Au,h^{-1}Pu \right\rangle|.
\end{equation}

Since $u\in L^2(\mathbb{T}^n)$ (uniformly in $h$), in order to ensure $L_2 u\in L^2(\mathbb{T}^n)$, we need $2l+2\leq 0$, which holds if and only if $l\leq -1$.  This is exactly what we assumed.  Then by Cauchy--Schwarz,
\[
|\left\langle L_2 u,u \right\rangle_{L^2(\mathbb{T}^n)}| \leq \|L_2 u\|_{L^2(\mathbb{T}^n)} \|u\|_{L^2(\mathbb{T}^n)} < \infty.
\]
Since $D\in\Psi^{m,l}_{2,h}(\mathcal{C})$ is microsupported on the $H_1$ orbit through $\zeta$, $\zeta\notin {}^2 \mathrm{WF}^{m,l}(u)$, and we know $H_1$ invariance, then $Du\in L^2(\mathbb{T}^n)$.
By construction,
\[
{}^2 \mathrm{WF}_{2l+1}'(A) \cap {}^2 \mathrm{WF}^{m-1/2,l}(u) = \emptyset.
\]
Therefore, since
\[
{}^2 \mathrm{WF}_{2l}'(L_1) = {}^2 \mathrm{WF}_{2l+2}'(L) \subset {}^2 \mathrm{WF}_{2l+1}'(A),
\]
$L_1\in\Psi^{2m-1,2l}_{2,h}(\mathcal{C})$ satisfies $|\left\langle L_1 u,u \right\rangle| < \infty$.  Finally, the last remaining term is controlled by our assumption that $Pu = O_{L^2(\mathbb{T}^n)}(h^\infty)$.

Since the RHS of Equation \eqref{eq:rearranged eqn} is bounded (as $h\rightarrow 0$), each term on the LHS is bounded.  The boundedness of $\|Cu\|^2$ demonstrates absence of ${}^2 \mathrm{WF}^{m,l}(u)$ on the microsupport of $C$.  Since $C$ is microsupported near the $\widetilde{H}_2$ flow-segment-ends opposite $\overline{\mathcal{O}}_{H_1}(\zeta)$ and since $\delta$ can be made arbitrarily small, we have proved that (lack of) ${}^2 \mathrm{WF}^{m,l}(u)$ spreads along each piece of $\widetilde{H}_2$ flow.
\end{proof}

\begin{remark}
In fact, to control the last term in \eqref{eq:rearranged eqn}, it suffices to assume that $Pu = O_{L^2(\mathbb{T}^n)}(h^s)$ for $s\geq\max(2m+1,2l+2)$,
since if $Pu = h^s g$ for $s\geq 2m+1$, $s\geq 2l+2$, and $g\in L^2(\mathbb{T}^n)$, then
\[
|\left\langle Au,h^{-1}Pu \right\rangle_{L^2(\mathbb{T}^n)}| = |\left\langle h^{s-1}Au,g \right\rangle_{L^2(\mathbb{T}^n)}| \leq \|h^{s-1}Au\|_{L^2(\mathbb{T}^n)} \|g\|_{L^2(\mathbb{T}^n)} < \infty.
\]
\end{remark}

\noindent Therefore, if we only wish to prove invariance of the \emph{graded} second wavefront ${}^2 \mathrm{WF}^{m,l}(u)$:

\begin{theorem}
Let $\mathcal{C}\subset T^* \mathbb{T}^n$ be a linear coisotropic submanifold.  Assume $P\in\widetilde{\Psi}^0_h(\mathbb{T}^n)$ has real valued principal symbol depending only on the fiber variables in $T^* \mathbb{T}^n$, and subprincipal symbol identically equal to zero.  Let $u \in L^2(\mathbb{T}^n)$.  Then for all $m$ and $l \leq -1$, ${}^2 \mathrm{WF}^{m,l}(u) \cap SN(\mathcal{C})$ propagates along the flow of $H_2$, if $u$ satisfies $Pu = O_{L^2(\mathbb{T}^n)}(h^s)$ for $s\geq\max(2m+1,2l+2)$.
\end{theorem}

The Hamiltonian flow of $P\in\Psi_h(\mathbb{T}^n)$ on $\mathcal{C}\subset T^* \mathbb{T}^n$ is described by quasi-periodic motion with respect to a set of frequencies $\overline{\omega}_1,\ldots,\overline{\omega}_n$.  By definition, these frequencies are the derivatives of $\sigma_\mathrm{pr}(P)(\xi)$ with respect each $\xi_j$, restricted to $\mathcal{C}$.  Hence, if all frequencies are irrationally related, then coisotropic regularity fills out the coisotropic, in the base variables.
We also consider the complementary case, in which $\overline{\omega}_i / \overline{\omega}_j \in \mathbb{Q}$ for some $i,j$.
Again, coisotropic regularity occurs on whole Hamiltonian orbits, but in this case, these orbits need not be dense.  Here coisotropic regularity is invariant under two separate flows, according to \autoref{secondary propagation}.

\subsection{Propagation Examples\label{ex:propag examples}}

We wish to apply our real principal type and secondary propagation theorems to quasimodes of the Laplacian $h^2 \Delta$.  However, note that \autoref{secondary propagation} is only valid for $h$-pseudodifferential operators with compactly supported symbols, belonging to the subalgebra $\widetilde{\Psi}^0_h(\mathbb{T}^n) \subset \Psi^0_h(\mathbb{T}^n)$.  Suppose $u = u_h\in L^2(\mathbb{T}^n)$ satisfies $Pu = O_{L^2(\mathbb{T}^n)}(h^\infty)$ for $P = h^2 \Delta - 1$.  Then
\begin{equation} \label{why this cutoff}
\mathrm{WF}_h(u) \subset \mathrm{char}(P) = \{|\xi| = 1\}.
\end{equation}
Clearly, $P\notin\widetilde{\Psi}^0_h(\mathbb{T}^n)$.  Let $\chi$ be a smooth and compactly supported function satisfying $\chi(x)\equiv 1$ in some neighborhood of the point $x=1$.  Then $P = \chi(h^2 \Delta)(h^2 \Delta - 1)\in\widetilde{\Psi}^0_h(\mathbb{T}^n)$ and $${}^2 \mathrm{WF}^{\infty,l}(u) \subset {}^2 \mathrm{char}(P).$$  Notice that we abuse notation and refer to the truncated operator also as $P$.  We may then use \autoref{secondary propagation} to study propagation of ${}^2 \mathrm{WF}^{\infty,l}(u)$ at $SN(\mathcal{C})$ for any linear coisotropic $\mathcal{C} \subset T^* \mathbb{T}^n$.

\begin{example}
Consider the Laplace operator $P = \left(h^2 \Delta_x - 1\right)/2$, and $\mathcal{C} = \{\xi'=(\xi_1,\xi_2)=(0,0)\} \subset T^*\mathbb{T}^4$.  We cut off $P$ as described above.  The Hamiltonian vector field determined by the principal symbol of $P$ is $\xi\cdot\partial_x$.  Coordinates near $SN (\mathcal{C})$ are $\rho_\mathrm{ff} = |\xi'|$, $\hat{\xi}' = \xi' / |\xi'|$, as well as $x$ and $\xi''$.  Then $\mathbf{H} = \xi''\cdot\partial_{x''} + \rho_\mathrm{ff}\left(\hat{\xi}'\cdot\partial_{x'}\right)$.  This is consistent with \eqref{Hamilton}.  For $Pu = 0$, invariance of ${}^2 \mathrm{WF}(u) \cap SN(\mathcal{C})$ under $H_1 = \xi''\cdot\partial_{x''}$ only gives propagation in directions not tangent to the leaves of the characteristic foliation, whereas invariance under $\widetilde{H}_2 = \hat{\xi}'\cdot\partial_{x'}$ gives propagation along the leaves only.


Elements of ${}^2 \mathrm{WF}(u) \cap SN(\mathcal{C})$ satisfy $|\xi''| = 1$ (since ${}^2 \mathrm{WF}(u)\subset {}^2 \mathrm{char}(P)$), and also satisfy $|\hat{\xi}'| = 1$.  For generic values of $\xi''$ and $\hat{\xi}'$ subject to these constraints, $H_1$ and $\widetilde{H}_2$ together flow to all of $\mathbb{T}^4$, but there are exceptional cases.

\end{example}

\begin{example}
For the same operator $P$ and the coisotropic $\mathcal{C} = \{\xi_1 + \xi_3 = \xi_2 + \xi_4 = 0\}$ (so $\mathbf{v}_1 = (1 \ 0 \ 1 \ 0)^t$, $\mathbf{v}_2 = (0 \ 1 \ 0 \ 1)^t$), define
\[
\hat{\xi}_1 = \frac{\xi_1 + \xi_3}{\sqrt{(\xi_1 + \xi_3)^2 + (\xi_2 + \xi_4)^2}}, \ \ \hat{\xi}_2 = \frac{\xi_2 + \xi_4}{\sqrt{(\xi_1 + \xi_3)^2 + (\xi_2 + \xi_4)^2}}
\]
(regarding $\xi_3$ and $\xi_4$ as free variables), and take $\rho_\mathrm{ff} = \sqrt{(\xi_1 + \xi_3)^2 + (\xi_2 + \xi_4)^2}$.  Then $\xi\cdot\partial_x$ lifts to
\[
\xi_3(\partial_{x_3} - \partial_{x_1}) + \xi_4(\partial_{x_4} - \partial_{x_2}) + \rho_\mathrm{ff}\left(\hat{\xi}_1 \partial_{x_1} + \hat{\xi}_2 \partial_{x_2}\right),
\]
with $H_1 = \xi_3(\partial_{x_3} - \partial_{x_1}) + \xi_4(\partial_{x_4} - \partial_{x_2})$ and $\widetilde{H}_2 = \hat{\xi}_1 \partial_{x_1} + \hat{\xi}_2 \partial_{x_2}$.  Again, this is consistent with \eqref{Hamilton}.  (In the formulation of \eqref{Hamilton}, for $d<j\leq n$, $\widetilde{x}_j = \mathbf{w}_j\cdot x$.  To match the coordinates above, we would choose $\mathbf{w}_3 = (0 \ 0 \ 1 \ 0)^t$, $\mathbf{w}_4 = (0 \ 0 \ 0 \ 1)^t$.)  For $Pu = 0$, note that elements of ${}^2 \mathrm{WF}(u) \cap SN(\mathcal{C})$ satisfy $\xi_3^2 + \xi_4^2 =\frac{1}{2}$.


\end{example}

\vspace{.5cm}

\footnotesize

\noindent \textsc{Department of Mathematics, University of Michigan, 530 Church Street, Ann Arbor, MI 48109} \\
\textit{E-mail address}: {\tt \href{mailto:kadakia@umich.edu}{kadakia@umich.edu}}

\end{document}

%% file: totsymb.pdf_tex
\begingroup%
  \makeatletter%
  \providecommand\color[2][]{%
    \errmessage{(Inkscape) Color is used for the text in Inkscape, but the package 'color.sty' is not loaded}%
    \renewcommand\color[2][]{}%
  }%
  \providecommand\transparent[1]{%
    \errmessage{(Inkscape) Transparency is used (non-zero) for the text in Inkscape, but the package 'transparent.sty' is not loaded}%
    \renewcommand\transparent[1]{}%
  }%
  \providecommand\rotatebox[2]{#2}%
  \ifx\svgwidth\undefined%
    \setlength{\unitlength}{336.4296875bp}%
    \ifx\svgscale\undefined%
      \relax%
    \else%
      \setlength{\unitlength}{\unitlength * \real{\svgscale}}%
    \fi%
  \else%
    \setlength{\unitlength}{\svgwidth}%
  \fi%
  \global\let\svgwidth\undefined%
  \global\let\svgscale\undefined%
  \makeatother%
  \begin{picture}(1,0.62756427)%
    \put(0,0){\includegraphics[width=\unitlength]{totsymb.pdf}}%
    \put(0.32743372,0.63320144){\color[rgb]{0,0,0}\makebox(0,0)[lt]{\begin{minipage}{0.28396305\unitlength}\centering $h$\end{minipage}}}%
    \put(0.85663715,0.28147482){\color[rgb]{0,0,0}\makebox(0,0)[lt]{\begin{minipage}{0.14520826\unitlength}\centering $\xi_2$\end{minipage}}}%
    \put(-0.00815876,0.04914153){\color[rgb]{0,0,0}\makebox(0,0)[lt]{\begin{minipage}{0.15488882\unitlength}\centering $\xi_1$\end{minipage}}}%
  \end{picture}%
\endgroup%

%% file: smear.pdf_tex
\begingroup%
  \makeatletter%
  \providecommand\color[2][]{%
    \errmessage{(Inkscape) Color is used for the text in Inkscape, but the package 'color.sty' is not loaded}%
    \renewcommand\color[2][]{}%
  }%
  \providecommand\transparent[1]{%
    \errmessage{(Inkscape) Transparency is used (non-zero) for the text in Inkscape, but the package 'transparent.sty' is not loaded}%
    \renewcommand\transparent[1]{}%
  }%
  \providecommand\rotatebox[2]{#2}%
  \ifx\svgwidth\undefined%
    \setlength{\unitlength}{390.4828295bp}%
    \ifx\svgscale\undefined%
      \relax%
    \else%
      \setlength{\unitlength}{\unitlength * \real{\svgscale}}%
    \fi%
  \else%
    \setlength{\unitlength}{\svgwidth}%
  \fi%
  \global\let\svgwidth\undefined%
  \global\let\svgscale\undefined%
  \makeatother%
  \begin{picture}(1,0.29465124)%
    \put(0,0){\includegraphics[width=\unitlength]{smear.pdf}}%
    \put(0.49722072,0.2994952){\color[rgb]{0,0,0}\makebox(0,0)[lt]{\begin{minipage}{0.19461183\unitlength}\centering $\zeta$\end{minipage}}}%
    \put(0.08019533,0.04232954){\color[rgb]{0,0,0}\makebox(0,0)[lt]{\begin{minipage}{1.02866226\unitlength}\centering $\exp(-\delta \widetilde{H}_2) \zeta$\end{minipage}}}%
    \put(-0.01779218,0.25812464){\color[rgb]{0,0,0}\makebox(0,0)[lt]{\begin{minipage}{0.21624641\unitlength}\centering $\mathcal{O}_{H_1}(\zeta)$\end{minipage}}}%
  \end{picture}%
\endgroup%